\documentclass[12pt]{article}
\usepackage{amssymb}
\usepackage{amsmath,amsthm}
\usepackage[latin1]{inputenc}
\usepackage{color}
\usepackage{graphicx}
\usepackage{tikz}
\usepackage{tkz-graph}
\usepackage{hyperref}
\usepackage{currfile}
\usepackage{datetime}
\usepackage{enumerate}
\hypersetup{colorlinks=true, urlcolor=blue}

\usepackage{epsfig}

 \newtheorem{remark}{Remark}

\newtheorem{observation}[remark]{Observation}

 \newtheorem{lemma}[remark]{Lemma}
 \newtheorem{theorem}[remark]{Theorem}
 \newtheorem{proposition}[remark]{Proposition}
 \newtheorem{corollary}[remark]{Corollary}
  
    \newtheorem{conjecture}[remark]{Conjecture}

\title{Strong resolving graphs: the realization and the characterization problems}

\date{}

\author{D. Kuziak$^{(1)}$,  M. L. Puertas$^{(2)}$, \\
J.  A. Rodr\'{\i}guez-Vel\'{a}zquez$^{(1)}$, I. G. Yero$^{(3)}$\\
\\
$^{(1)}${\small Departament d'Enginyeria Inform\`atica i Matem\`atiques,}\\
{\small Universitat Rovira i Virgili,} {\small Av. Pa\"{\i}sos Catalans 26, 43007 Tarragona, Spain.}\\
{\small dorota.kuziak\@@urv.cat, juanalberto.rodriguez\@@urv.cat}\\
$^{(2)}$
{\small Departamento de Matem\'{a}ticas, Universidad de Almer\'{i}a}\\ {\small  Ctra. Sacramento s/n, La Ca\~nada de San Urbano, 04120 Almer\'{i}a, Spain}\\
{\small mpuertas\@@ual.es}
\\
$^{(3)}${\small Departamento de Matem\'aticas, Escuela Polit\'ecnica Superior de Algeciras}\\
{\small Universidad de C\'adiz,} {\small Av. Ram\'on Puyol s/n, 11202 Algeciras, Spain.}\\
{\small ismael.gonzalez\@@uca.es}\\
}

\begin{document}
\maketitle

\begin{abstract}
The  strong resolving graph $G_{SR}$ of a connected graph $G$ was introduced in [Discrete Applied Mathematics 155 (1) (2007) 356--364] as a tool to study the strong metric dimension of $G$. Basically, it was shown  that the problem of finding the strong metric dimension of $G$ can be transformed to the problem of finding the vertex cover number of  $G_{SR}$.
Since then, several articles dealing with this subject have been published.
In this paper, we survey the state of knowledge on  the strong resolving graph  and also derive some new results.
\end{abstract}




\section{Introduction}


Graphs are basic combinatorial structures, and transformations  of structures are  fundamental  to the development of mathematics.  Particularly, in graph theory, some elementary transformations generate a new graph from an original one by some simple local changes, such as addition or deletion of a vertex or of an edge, merging and splitting of vertices, edge contraction, etc. Other advanced transformations create a new graph from the original one by complex changes, such as complement graph, line graph, total graph, graph power, dual graph, strong resolving graph, etc.

Some of these transformations of graphs emerged as a natural tool to solve practical problems. In other cases, the problem of finding a specific parameter of a graph has become the problem of finding another parameter of another graph obtained from the original one.  This is the case of the strong resolving graph $G_{SR}$ of a connected graph $G$ which was introduced by Oellermann and Peters-Fransen in \cite{Oellermann2007} as a tool to study the strong metric dimension of $G$. Basically, it was shown  that the problem of finding the strong metric dimension of $G$ can be transformed to the problem of finding the vertex cover number of  $G_{SR}$.
Since then, several articles dealing with the strong resolving graph have been published. However, in almost all these works the results related to the strong resolving graph are not explicit, as they implicitly appear as a part of the proofs of  main results concerning the strong metric dimension. In this sense, this interesting construction has passed in front of researchers's eyes without the attention that should require.
In this paper, we make an attempt of motivating the community of graph theorists to have a look into this direction and take more in consideration this construction. Accordingly, herein we survey the state of knowledge on  the strong resolving graph  and also derive some new results.

For a graph transformation, there are two general problems \cite{Gunbaum1969}, which we shall formulate in terms of strong resolving graphs:

\begin{itemize}
\item  \textbf{Realization Problem.}\footnote{This problem was called Determination Problem in \cite{Gunbaum1969}.} Determine which graphs have a given graph as their strong resolving graphs.
\item  \textbf{Characterization Problem.} Characterize those graphs that are strong resolving graphs of some graphs.
\end{itemize}

The majority of results presented in this paper concerns the above mentioned problems. Basically,  we focus on the following graph equation
\begin{equation}\label{SRG-GraphEquation}
G_{SR}\cong H,
\end{equation}
\emph{i.e.}, the goal is to find all pairs of graphs $G$ and $H$ satisfying \eqref{SRG-GraphEquation}.

The remainder of the paper is structured as follows. Subsection \ref{SubSectTerminology} covers general notation and terminology. Subsection \ref{SubSectStrongMetricDimension} is devoted to introduce the strong metric dimension, whereas Subsection \ref{SubSectStrongResolvingGraph} introduces the strong resolving graph. In Section \ref{SectionDetermination problem} we study the realization problem for some specific families of graphs, while in Section \ref{SectionCharacterization problem for product graphs} we collect the known results related to the characterization problem of product graphs. We close our exposition with a collection of open problems to be dealt with. In order to gain more completeness of this work, we include or improve the proofs of some results which are remarkable for the topic, although the main part of them are already published in some journals.

\subsection{Notation and Terminology}\label{SubSectTerminology}

We continue by establishing the basic terminology and notations which is used throughout this work. For the sake of completeness we refer the reader to the books \cite{Diestel2005,Hammack2011,West1996}. Graphs considered herein are undirected, finite and contain neither loops nor multiple edges. Let $G$ be a graph of order $n = |V(G)|$. A graph is nontrivial if $n\ge 2$. We use the notation $u\sim v$ for two adjacent vertices $u$ and $v$ of $G$. For a vertex $v$ of $G$, $N_G(v)$ denotes the set of neighbors that $v$ has in $G$, \emph{i.e.},  $N_G(v)=\{u\in V(G):\; u\sim v\}$. The set $N_G(v)$ is called the \emph{open neighborhood of a vertex} $v$ in $G$ and $N_G[v]=N_G(v)\cup \{v\}$ is called the \emph{closed neighborhood of a vertex} $v$ in $G$. The \emph{degree} of a vertex $v$ of $G$ is denoted by $\delta_G(v)$, \emph{i.e.}, $\delta_G(v)=|N_G(v)|$. The \emph{open neighborhood of a set} $S$ of vertices of $G$ is $N_G(S)=\bigcup_{v\in S}N_G(v)$ and the \emph{closed neighborhood of} $S$ is $N_G[S]=N_G(S)\cup S$.

We use the notation $K_n$, $C_n$, $P_n$, and $N_n$ for the \emph{complete graph}, \emph{cycle}, \emph{path}, and \emph{empty graph}, respectively. Moreover, we write $K_{s,t}$ for the \emph{complete bipartite graph} of order $s+t$ and in particular case $K_{1,n}$ for the \emph{star} of order $n+1$. Let $T$ be a \emph{tree}, a vertex of degree one in $T$ is called a \emph{leaf} and the number of leaves in $T$ is denoted by $l(T)$.

The \emph{distance} between two vertices $u$ and $v$, denoted by $d_{G}(u,v)$, is the length of a shortest path between $u$ and $v$ in $G$. The \emph{diameter}, $D(G)$, of $G$ is the longest distance between any two vertices of $G$ and two vertices $u,v\in V(G)$ such that $d_G(u,v)=D(G)$ are called \emph{diametral}. If $G$ is not connected, then we assume that the distance between any two vertices belonging to different components of $G$ is infinity and, thus, its diameter is $D(G)=\infty$.\label{g graph} A graph $G$ is $2$-\emph{antipodal} if for each vertex $x\in V(G)$ there exists exactly one vertex $y\in V(G)$ such that $d_G(x,y)=D(G)$. For instance, even cycles and hypercubes are $2$-antipodal graphs.

We recall that the \emph{complement} of $G$ is the graph $G^c$ with the same vertex set as $G$ and $uv\in E(G^c)$ if and only if $uv\notin E(G)$. The \emph{subgraph induced by a set} $X$ is denoted by $\langle X \rangle$\label{g ind subgraph}. A vertex of a graph is a {\em simplicial vertex} if the subgraph induced by its neighbors is a complete graph. Given a graph $G$, we denote by $\sigma(G)$ the set of simplicial vertices of $G$.

A \emph{clique} in $G$ is a set of pairwise adjacent vertices. The \emph{clique number} of $G$, denoted by $\omega(G)$\label{g clique}, is the number of vertices in a maximum clique in $G$. Two distinct vertices $u$, $v$ are called \emph{true twins} if $N_G[u] = N_G[v]$. In this sense, a vertex $x$ is a \emph{twin} if there exists $y\ne x$ such that they are true twins. We say that $X\subset V(G)$ is a \emph{twin-free clique} in $G$ if the subgraph induced by $X$ is a clique and for every $u,v\in X$ it follows $N_G[u]\ne N_G[v]$, \emph{i.e.}, the subgraph induced by $X$ is a clique and it contains no true twins. The \emph{twin-free clique number} of $G$, denoted by $\varpi(G)$, is the maximum cardinality among all twin-free cliques in $G$. So, $\omega(G)\ge \varpi(G)$. We refer to a $\varpi(G)$-set in a graph $G$ as a twin-free clique  of cardinality $\varpi(G)$.
Figure \ref{ex twins} shows examples of basic concepts such as true twins and twin-free clique.

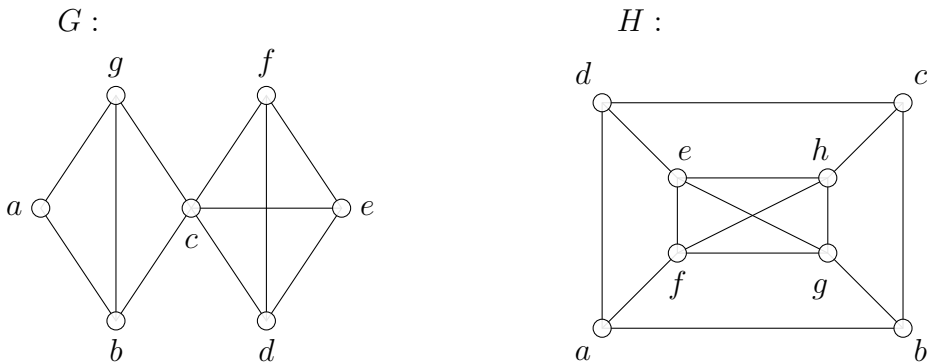
\begin{figure}[ht]
\centering
\begin{tabular}{cccccc}
\begin{tikzpicture}

\draw(0,0) -- (1,1.5) -- (2,0) -- (1,-1.5) -- cycle;
\draw(2,0) -- (3,1.5) -- (4,0) -- (3,-1.5) -- cycle;
\draw(1,1.5) -- (1,-1.5);
\draw(3,1.5) -- (3,-1.5);
\draw(2,0) -- (4,0);

\filldraw[fill opacity=0.9,fill=white]  (0,0) circle (0.12cm);
\filldraw[fill opacity=0.9,fill=white]  (1,1.5) circle (0.12cm);
\filldraw[fill opacity=0.9,fill=white]  (2,0) circle (0.12cm);
\filldraw[fill opacity=0.9,fill=white]  (1,-1.5) circle (0.12cm);
\filldraw[fill opacity=0.9,fill=white]  (3,1.5) circle (0.12cm);
\filldraw[fill opacity=0.9,fill=white]  (4,0) circle (0.12cm);
\filldraw[fill opacity=0.9,fill=white]  (3,-1.5) circle (0.12cm);

\node at (0.5,2.5) {$G:$};

\node [left] at (-0.1,0) {$a$};
\node [below] at (1,-1.6) {$b$};
\node [below] at (2,-0.2) {$c$};
\node [below] at (3,-1.6) {$d$};
\node [right] at (4.1,0) {$e$};
\node [above] at (3,1.6) {$f$};
\node [above] at (1,1.6) {$g$};

\end{tikzpicture} & & \hspace*{0.7cm} & &
\begin{tikzpicture}
\draw(0,-1.5) -- (0,1.5) -- (4,1.5) -- (4,-1.5) -- cycle;
\draw(1,-0.5) -- (1,0.5) -- (3,0.5) -- (3,-0.5) -- cycle;
\draw(1,-0.5) -- (3,0.5);
\draw(1,0.5) -- (3,-0.5);
\draw(0,1.5) -- (1,0.5);
\draw(0,-1.5) -- (1,-0.5);
\draw(4,1.5) -- (3,0.5);
\draw(4,-1.5) -- (3,-0.5);

\filldraw[fill opacity=0.9,fill=white]  (0,-1.5) circle (0.12cm);
\filldraw[fill opacity=0.9,fill=white]  (0,1.5) circle (0.12cm);
\filldraw[fill opacity=0.9,fill=white]  (4,1.5) circle (0.12cm);
\filldraw[fill opacity=0.9,fill=white]  (4,-1.5) circle (0.12cm);
\filldraw[fill opacity=0.9,fill=white]  (1,-0.5) circle (0.12cm);
\filldraw[fill opacity=0.9,fill=white]  (1,0.5) circle (0.12cm);
\filldraw[fill opacity=0.9,fill=white]  (3,0.5) circle (0.12cm);
\filldraw[fill opacity=0.9,fill=white]  (3,-0.5) circle (0.12cm);

\node at (0.5,2.6) {$H:$};

\node [below left] at (0,-1.6) {$a$};
\node [below right] at (4,-1.5) {$b$};
\node [above right] at (4,1.6) {$c$};
\node [above left] at (0,1.6) {$d$};
\node [above] at (1.1,0.6) {$e$};
\node [below] at (1,-0.6) {$f$};
\node [below] at (2.9,-0.7) {$g$};
\node [above] at (2.9,0.6) {$h$};

\end{tikzpicture} \\
\end{tabular}
\caption{The set $\{d,e,f\}\subset V(G)$ is composed by true twin vertices in $G$. Notice that $b$ and $g$ are true twin vertices in $G$ which are not simplicial, while $f$ and $d$ are true twin and simplicial vertices. The set $\{e,f,g,h\}\subset V(H)$ is a twin-free clique in $H$.}
\label{ex twins}
\end{figure}

For the remainder of the paper, definitions will be introduced whenever a concept is needed.

\subsection{Strong Metric Dimension of Graphs}\label{SubSectStrongMetricDimension}



A vertex $w\in V(G)$ \emph{strongly resolves} two different vertices $u,v\in V(G)$ if $d_G(w,u)=d_G(w,v)+d_G(v,u)$ or $d_G(w,v)=d_G(w,u)+d_G(u,v)$, \emph{i.e.}, there exists some shortest $w-u$ path containing $v$ or some shortest $w-v$ path containing $u$. A set $S$ of vertices in a connected graph $G$ is a \emph{strong metric generator} for $G$ if every two vertices of $G$ are strongly resolved by some vertex of $S$. The smallest cardinality of a strong metric generator for $G$ is called the \emph{strong metric dimension} and is denoted by $dim_s(G)$\label{g dims}. A \emph{strong metric basis} of $G$ is a strong metric generator for $G$ of cardinality $dim_s(G)$.

Several researches on the strong metric dimension of graphs have recently been developed. For instance, the trivial bounds $1\le dim_s(G)\le n-1$ are known from the first works as well as characterizations on whether they are tight. Moreover, it has been noticed that the strong metric dimension of several graphs can be straightforwardly computed for some basic examples which we next remark.

\begin{observation}\label{values-sdim-basic}$\,$
\begin{enumerate}[{\rm(a)}]
  \item $dim_s(G)=1$ if and only if $G$ is isomorphic to the path $P_n$ on $n\ge 2$ vertices.
  \item $dim_s(G)=n-1$ if and only if $G$ is isomorphic to the complete graph $K_n$ on $n\ge 2$ vertices.
  \item For any cycle $C_n$ of order $n$, $dim_s(C_n) = \lceil n/2\rceil$.
  \item For any tree $T$ with $l(T)$ leaves, $dim_s(T) = l(T)-1$.
  \item For any complete bipartite graph $K_{r,t}$, $dim_s(K_{r,t}) = r+t-2$.
\end{enumerate}
\end{observation}


The strong metric dimension is a relatively new parameter (defined in 2004). Since then, this parameter has been investigated for several classes of graphs. For instance, we cite the works on Cartesian product graphs \cite{Kratica2012,Oellermann2007,RodriguezVelazquez2014a}, Cartesian sum graphs \cite{Kuziak2014b}, corona graphs \cite{Kuziak2013}, direct product graphs \cite{Kuziak2016a,RodriguezVelazquez2014a}, strong product graphs \cite{Kuziak2013c}, lexicographical product graphs \cite{Kuziak2014}, Cayley graphs \cite{Oellermann2007}, Sierpi\'{n}ski graphs \cite{Rodriguez-Velazquez2016}, distance-hereditary graphs \cite{May2011}, and convex polytopes \cite{Kratica2012a}. Also, some Nordhaus-Gaddum type results for the strong metric dimension of a graph and its complement are known \cite{Yi2013}. Besides the theoretical results related to the strong metric dimension, a mathematical programming model \cite{Kratica2012a} and metaheuristic approaches \cite{Kratica2008,Mladenovic2012a} for finding this parameter have been developed. Some complexity and approximation results are also known from the works \cite{Oellermann2007} and \cite{DasGupta2016}, respectively. On the other hand, a fractional version of the strong metric dimension has been studied in \cite{Kang2016,Kang2016A,Kang2013}. In these three works the strong resolving graph is also used as an important tool. For more information we refer the reader to the survey \cite{Kratica2014} and the Ph.D. thesis \cite{Kuziak2014a}.


\subsection{The Strong Resolving Graph}\label{SubSectStrongResolvingGraph}

In \cite{Oellermann2007}, the authors have developed an approach which transforms the problem of finding the strong metric dimension of a graph to the problem of computing the vertex cover number of some other related graph. This relationship arises in connected with the following definitions.

A vertex $u$ of $G$ is \emph{maximally distant} from $v$ if for every vertex $w\in N_G(u)$, $d_G(v,w)\le d_G(u,v)$. We denote by $M_G(v)$ the set of vertices of $G$ which are maximally distant from $v$. The collection of all vertices of $G$ that are maximally distant from some vertex of the graph is called the {\em boundary} of the graph, see \cite{Brevsar2008,Caceres2005}, and is denoted by $\partial(G)$\footnote{In fact, the boundary $\partial(G)$ of a graph was defined first in \cite{Chartrand2003d} as the subgraph of $G$ induced by the set mentioned in our work with the same notation. We follow the approach of \cite{Brevsar2008,Caceres2005} where the boundary of the graph is just the subset of the boundary vertices defined in this article.}\label{g boundary}. If $u$ is maximally distant from $v$ and $v$ is maximally distant from $u$, then $u$ and $v$ are \emph{mutually maximally distant} (from now on MMD for short).

\begin{remark}
$\partial(G)=\{u\in V(G):$ there exists $v\in V(G)$ such that $u,v$ are MMD$\}$.
\end{remark}
\begin{proof}
On the one hand, if $u$ is maximally distant from $v$, and $v$ is not maximally distant from $u$, then $v$ has a neighbor $v_1$, such that $d_G(v_1, u) > d_G(v,u)$, \emph{i.e.}, $d_G(v_1, u) =d_G(v,u)+1$. It is easily seen that $u$ is maximally distant from $v_1$. If $v_1$ is not maximally distant from $u$, then $v_1$ has a neighbor $v_2$, such that $d_G(v_2, u) >d_G(v_1,u)$. Continuing in this manner we construct a sequence of vertices $v_1,v_2, \ldots$ such that $d_G(v_{i+1}, u) > d_G(v_i, u)$ for every $i$. Since $G$ is finite this sequence terminates with some $v_k$. Thus for all neighbors $x$ of $v_k$ we have $d_G(v_k,u) \ge d_G(x,u)$, and so $v_k$ is maximally distant from $u$ and $u$ is maximally distant from $v_k$. Hence every boundary vertex belongs to the set $S=\{u\in V(G):$ there exists $v\in V(G)$ such that $u,v$ are MMD$\}$. On the other hand, certainly every vertex of $S$ is a boundary vertex.
\end{proof}

For some basic graph classes, such as complete graphs $K_n$, complete bipartite graphs $K_{r,s}$, cycles $C_n$ and hypercube graphs $Q_k$, the boundary is simply the whole vertex set. It is not difficult to see that this property also holds for all $2$-antipodal graphs and for all vertex transitive graphs. Notice that the boundary of a tree consists of its leaves. Also, it is readily seen that every simplical vertex is a boundary vertex, that is $\sigma(G)\subseteq \partial(G)$.

Figure \ref{mutually and boundary} shows examples of basic concepts such as maximally distant vertices, MMD vertices and boundary. As a direct consequence of the definition of MMD vertices, we have the following.

\begin{figure}[ht]
\centering
\begin{tikzpicture}

\draw(0,0) -- (1,1.5) -- (2,0) -- (1,-1.5) -- cycle;
\draw(2,0) -- (3,1.5) -- (4,0) -- (3,-1.5) -- cycle;
\draw(4,0) -- (5,1.5) -- (6,0) -- (5,-1.5) -- cycle;
\draw(1,1.5) -- (1,-1.5);
\draw(5,1.5) -- (5,-1.5);
\draw(4,0) -- (6,0);

\filldraw[fill opacity=0.9,fill=white]  (0,0) circle (0.12cm);
\filldraw[fill opacity=0.9,fill=white]  (1,1.5) circle (0.12cm);
\filldraw[fill opacity=0.9,fill=white]  (2,0) circle (0.12cm);
\filldraw[fill opacity=0.9,fill=white]  (1,-1.5) circle (0.12cm);
\filldraw[fill opacity=0.9,fill=white]  (3,1.5) circle (0.12cm);
\filldraw[fill opacity=0.9,fill=white]  (4,0) circle (0.12cm);
\filldraw[fill opacity=0.9,fill=white]  (3,-1.5) circle (0.12cm);
\filldraw[fill opacity=0.9,fill=white]  (5,1.5) circle (0.12cm);
\filldraw[fill opacity=0.9,fill=white]  (6,0) circle (0.12cm);
\filldraw[fill opacity=0.9,fill=white]  (5,-1.5) circle (0.12cm);

\node [left] at (-0.1,0) {$a$};
\node [below] at (1,-1.6) {$b$};
\node [below] at (2,-0.2) {$c$};
\node [below] at (3,-1.6) {$d$};
\node [below] at (4,-0.2) {$e$};
\node [below] at (5,-1.6) {$f$};
\node [right] at (6.1,0) {$g$};
\node [above] at (5,1.6) {$h$};
\node [above] at (3,1.6) {$i$};
\node [above] at (1,1.6) {$j$};

\end{tikzpicture}
\caption{The set $\{a,f,g,h\}$ is composed by simplicial vertices and its elements are MMD between them. Also, $b$ and $j$ ($d$ and $i$) are MMD. Thus, the boundary of $G$ is $\partial(G)=\{a,b,d,f,g,h,i,j\}$. Now, $M_G(d)=\{a,f,g,h,i\}$ is the set of vertices which are maximally distant from $d$. Nevertheless, the vertex $d$ is maximally distant only from the vertex $i$.}
\label{mutually and boundary}
\end{figure}
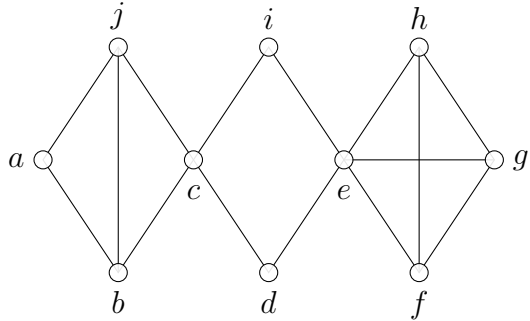

\begin{remark}
For every pair of MMD vertices $x,y$ of a connected graph $G$ and for every strong metric basis $S$ of $G$, it follows that $x\in S$ or $y\in S$.
\end{remark}

By using the concepts of boundary of a graph and MMD vertices, the notion of strong resolving graph was introduced in \cite{Oellermann2007} in the following way.  The \emph{strong resolving graph} of $G$ has vertex set of $V(G)$ and two vertices $u,v$ are adjacent if and only if $u$ and $v$ are MMD in $G$. Observe that the vertices of the set $V(G)-\partial(G)$ are isolated vertices in the strong resolving graph. According to this fact, in this work we use two slightly different versions of it, which are next stated.

The first version is denoted as $G_{SR}$\label{g GSR} while the second one is denoted by $G_{SR+I}$. The graph $G_{SR}$ has vertex set $\partial(G)$ and $G_{SR+I}$ has vertex set $V(G)$. Clearly, the difference between $G_{SR}$ and $G_{SR+I}$ is the existence of isolated vertices in $G_{SR+I}$, when $V(G)-\partial(G)\ne \emptyset$ and notice that the graph $G_{SR+I}$ coincides with the original definition presented in \cite{Oellermann2007}. The concept of the strong resolving graph $G_{SR}$ is used in this work rather than that of $G_{SR+I}$. The main reason of this fact is related to have a simpler notation and more clarity while proving the results. Figure \ref{GSR and GSR+I} shows the strong resolving graphs $G_{SR}$ and $G_{SR+I}$ of the graph $G$ illustrated in Figure \ref{mutually and boundary}.

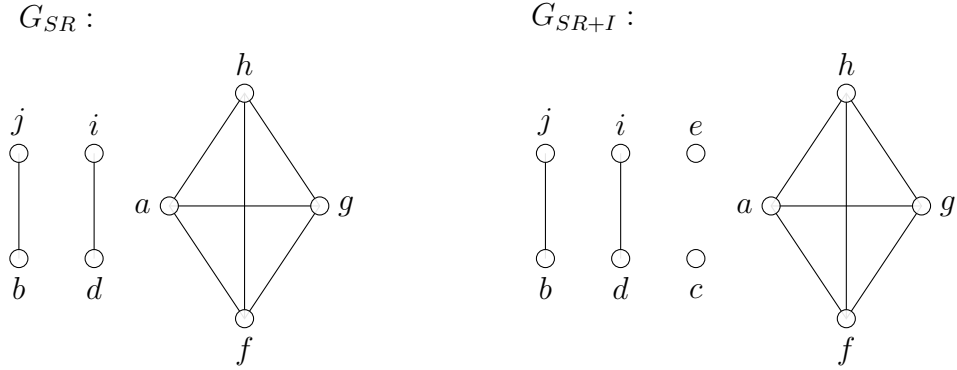
\begin{figure}[ht]
\centering
\begin{tikzpicture}

\draw(2,0) -- (3,1.5) -- (4,0) -- (3,-1.5) -- cycle;
\draw(0,0.7) -- (0,-0.7);
\draw(1,0.7) -- (1,-0.7);
\draw(3,1.5) -- (3,-1.5);
\draw(2,0) -- (4,0);

\filldraw[fill opacity=0.9,fill=white]  (0,0.7) circle (0.12cm);
\filldraw[fill opacity=0.9,fill=white]  (0,-0.7) circle (0.12cm);
\filldraw[fill opacity=0.9,fill=white]  (1,0.7) circle (0.12cm);
\filldraw[fill opacity=0.9,fill=white]  (2,0) circle (0.12cm);
\filldraw[fill opacity=0.9,fill=white]  (1,-0.7) circle (0.12cm);
\filldraw[fill opacity=0.9,fill=white]  (3,1.5) circle (0.12cm);
\filldraw[fill opacity=0.9,fill=white]  (4,0) circle (0.12cm);
\filldraw[fill opacity=0.9,fill=white]  (3,-1.5) circle (0.12cm);

\node at (0.5,2.5) {$G_{SR}:$};

\node [below] at (0,-0.8) {$b$};
\node [below] at (1,-0.8) {$d$};
\node [left] at (1.9,0) {$a$};
\node [below] at (3,-1.6) {$f$};
\node [right] at (4.1,0) {$g$};
\node [above] at (3,1.6) {$h$};
\node [above] at (1,0.8) {$i$};
\node [above] at (0,0.8) {$j$};


\draw(10,0) -- (11,1.5) -- (12,0) -- (11,-1.5) -- cycle;
\draw(7,0.7) -- (7,-0.7);
\draw(8,0.7) -- (8,-0.7);
\draw(11,1.5) -- (11,-1.5);
\draw(10,0) -- (12,0);

\filldraw[fill opacity=0.9,fill=white]  (7,0.7) circle (0.12cm);
\filldraw[fill opacity=0.9,fill=white]  (7,-0.7) circle (0.12cm);
\filldraw[fill opacity=0.9,fill=white]  (8,0.7) circle (0.12cm);
\filldraw[fill opacity=0.9,fill=white]  (10,0) circle (0.12cm);
\filldraw[fill opacity=0.9,fill=white]  (8,-0.7) circle (0.12cm);
\filldraw[fill opacity=0.9,fill=white]  (11,1.5) circle (0.12cm);
\filldraw[fill opacity=0.9,fill=white]  (12,0) circle (0.12cm);
\filldraw[fill opacity=0.9,fill=white]  (11,-1.5) circle (0.12cm);
\filldraw[fill opacity=0.9,fill=white]  (9,-0.7) circle (0.12cm);
\filldraw[fill opacity=0.9,fill=white]  (9,0.7) circle (0.12cm);

\node at (7.5,2.5) {$G_{SR+I}:$};

\node [below] at (7,-0.8) {$b$};
\node [below] at (8,-0.8) {$d$};
\node [below] at (9,-0.9) {$c$};
\node [above] at (9,0.8) {$e$};
\node [left] at (9.9,0) {$a$};
\node [below] at (11,-1.6) {$f$};
\node [right] at (12.1,0) {$g$};
\node [above] at (11,1.6) {$h$};
\node [above] at (8,0.8) {$i$};
\node [above] at (7,0.8) {$j$};

\end{tikzpicture}
\caption{$G_{SR}$ and $G_{SR+I}$ of the graph $G$ illustrated in Figure \ref{mutually and boundary}.}
\label{GSR and GSR+I}
\end{figure}

There are several families of graphs for which the strong resolving graph can be relatively easily described. We next state some of these here.

\par\bigskip

\par\bigskip

\begin{observation}\label{observation1}$\,$
\begin{enumerate}[{\rm(a)}]
\item If $\partial(G)=\sigma(G)$, then $G_{SR}\cong K_{|\partial(G)|}$. In particular, $(K_n)_{SR}\cong K_n$ and for any tree $T$, $T_{SR}\cong K_{l(T)}$.
\item For any $2$-antipodal graph $G$ of order $n$, $G_{SR}\cong \bigcup_{i=1}^{\frac{n}{2}} K_2$. In particular, $(C_{2k})_{SR}\cong \bigcup_{i=1}^{k} K_2$.
\item For odd cycles $(C_{2k+1})_{SR}\cong C_{2k+1}$.
\item For any complete $k$-partite graph $G=K_{p_1,p_2,\dots ,p_k}$ such that $p_i\ge 2$, $i\in\{1,2,\dots ,k\}$, $G_{SR}\cong\bigcup_{i=1}^{k}K_{p_i}$.
\end{enumerate}
\end{observation}

Recall that a set $S$ of vertices of $G$ is a \emph{vertex cover} of $G$ if every edge of $G$ is incident with at least one vertex of $S$. The \emph{vertex cover number} of $G$, denoted by $\beta(G)$\label{g vertex cover}, is the smallest cardinality of a vertex cover of $G$. We refer to a $\beta(G)$-set in a graph $G$ as a vertex cover  of cardinality $\beta(G)$.

Oellermann and Peters-Fransen \cite{Oellermann2007} showed that the problem of finding the strong metric dimension of a connected graph $G$ can be transformed to the problem of finding the vertex cover number of $G_{SR+I}$.

\begin{theorem}{\em \cite{Oellermann2007}}\label{th oellermann1}
For any connected graph $G$, $dim_s(G) = \beta(G_{SR+I}).$
\end{theorem}

Now, it is readily seen that $\beta(G_{SR+I})=\beta(G_{SR})$. Therefore, an analogous theorem to the one above can be stated by using $G_{SR}$ instead of $G_{SR+I}$.

\begin{theorem}\label{th oellermann}
For any connected graph $G$, $dim_s(G) = \beta(G_{SR}).$
\end{theorem}

Figure \ref{transformation} illustrates this theorem, which has proved its high usefulness in several situations.

\begin{figure}[ht]
\centering
\begin{tikzpicture}
\draw(0,1.3) -- (1.1,0.4) -- (0.7,-1) -- (-0.7,-1) -- (-1.1,0.4) -- cycle;
\draw(-1.1,0.4) -- (-2.2,0.4);
\draw(1.1,0.4) -- (2,1.1);
\draw(1.1,0.4) -- (2,-0.3);

\filldraw[fill opacity=0.9,fill=white]  (0,1.3) circle (0.12cm);
\filldraw[fill opacity=0.9,fill=white]  (1.1,0.4) circle (0.12cm);
\filldraw[fill opacity=0.9,fill=white]  (0.7,-1) circle (0.12cm);
\filldraw[fill opacity=0.9,fill=white]  (-0.7,-1) circle (0.12cm);
\filldraw[fill opacity=0.9,fill=white]  (-1.1,0.4) circle (0.12cm);

\filldraw[fill opacity=0.9,fill=white]  (-2.2,0.4) circle (0.12cm);
\filldraw[fill opacity=0.9,fill=white]  (2,1.1) circle (0.12cm);
\filldraw[fill opacity=0.9,fill=white]  (2,-0.3) circle (0.12cm);

\node at (-2.2,2.2) {$G:$};

\node [above] at (0,1.4) {$a$};
\node [above] at (1.1,0.5) {$b$};
\node [right] at (2.1,1.1) {$c$};
\node [right] at (2.1,-0.3) {$d$};
\node [below] at (0.7,-1.1) {$e$};
\node [below] at (-0.7,-1.1) {$f$};
\node [above] at (-1.1,0.5) {$g$};
\node [left] at (-2.3,0.5) {$h$};


\draw(6,1.3) -- (7.1,0.4) -- (6.7,-1) -- (5.3,-1) -- (4.9,0.4) -- cycle; 
\draw(6.2,0) -- (4.9,0.4);
\draw(6.2,0) -- (5.3,-1);
\draw(6.2,0) -- (6,1.3);

\filldraw[fill opacity=0.9,fill=white]  (6,1.3) circle (0.12cm);
\filldraw[fill opacity=0.9,fill=white]  (7.1,0.4) circle (0.12cm);
\filldraw[fill opacity=0.9,fill=white]  (6.7,-1) circle (0.12cm);
\filldraw[fill opacity=0.9,fill=white]  (5.3,-1) circle (0.12cm);
\filldraw[fill opacity=0.9,fill=white]  (4.9,0.4) circle (0.12cm);
\filldraw[fill opacity=0.9,fill=white]  (6.2,0) circle (0.12cm);

\node at (4.9,2.2) {$G_{SR}:$};

\node [right] at (7.2,0.4) {$a$};
\node [left] at (4.8,0.4) {$c$};
\node [right] at (6.3,0) {$d$};
\node [below] at (6.7,-1.1) {$e$};
\node [above] at (6,1.4) {$f$};
\node [below] at (5.3,-1.1) {$h$};

\end{tikzpicture}
\caption{The set $\{a,c,d,h\}\subset V(G)$ forms a strong metric basis of $G$. Also, the set $\{a,c,d,h\}\subset V(G_{SR})$ is a vertex cover of $G_{SR}$. Thus, $dim_s(G)=\beta(G_{SR})=4$.}
\label{transformation}
\end{figure}
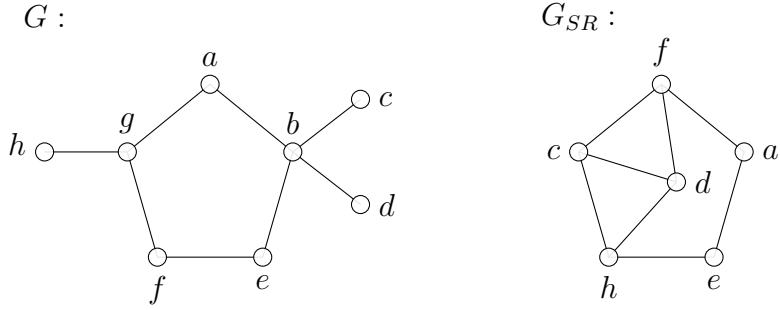

Recall that the largest cardinality of a set of vertices of $G$, no two of which are adjacent, is called the \emph{independence number} of $G$ and is denoted by $\alpha(G)$\label{g independence}. We refer to an $\alpha(G)$-set in a graph $G$ as an independent set of cardinality $\alpha(G)$. The following well-known result, due to Gallai \cite{Gallai1959}, states the relationship between the independence number and the vertex cover number of a graph.

\begin{theorem}{\em\cite{Gallai1959}}{\rm (Gallai, 1959)}\label{th gallai}
For any graph $G$ of order $n$, $\alpha(G)+\beta(G) = n.$
\end{theorem}

Thus, for any graph $G$, by using Theorems \ref{th oellermann} and \ref{th gallai}, we immediately obtain that
\begin{equation*}
dim_s(G) = |\partial(G)| - \alpha(G_{SR}) \label{oellermann-gallai}
\end{equation*}

\section{Realization Problem}\label{SectionDetermination problem}

In this section we study the realization problem for some specific families of graphs, \emph{i.e.}, we study the graph equation $G_{SR}\cong H$ where $H$ is isomorphic to $K_n$, $K_{1,r}$, $C_n$, $P_n$ and $G^c$. In addition, the characterization problem of graphs of diameter two is considered.
We begin with the characterization of graphs whose strong resolving graph is complete. To this end, we need the following two lemmas.

\begin{lemma}[\cite{Muller2008,Muller2011}]\label{lem:extension}
Each shortest path in a graph $G$ extends to a shortest path between two boundary vertices.
\end{lemma}

\begin{lemma}\label{lem:no_extreme}
Let $G$ be a graph and let $v\in \partial(G)\setminus \sigma(G)$. Then there exist $a,b\in \partial(G)\setminus\{v\}$ such that $va, vb$ are not edges of $G_{SR}$.
\end{lemma}

\begin{proof}
Let $v_1,v_2\in N(v)$ be such that $d(v_1,v_2)=2$. Then  $P=v_1 v v_2$ is a shortest path and, by Lemma~\ref{lem:extension}, there exist $a,b\in \partial (G)\setminus\{v\}$ and a shortest path between them that extends $P$. So $v$ lays in a shortest path between $a$ and $b$ and, in particular, $v$ is not maximally distance from any of them. This means that $v$ is not a neighbor of $a$ nor $b$ in $G_{SR}$.
\end{proof}

With these tools we obtain the following characterization.

\begin{theorem}\label{theorem:complete}
Let $G$ be a connected graph. Then $G_{SR}\cong K_{|\partial(G)|}$ if and only if $\partial(G)=\sigma(G)$.
\end{theorem}

\begin{proof}
If $\partial(G)=\sigma(G)$, it is clear that $G_{SR}\cong K_{|\partial(G)|}$. Conversely, assume now that $\sigma(G)\varsubsetneq \partial(G)$ and let $v\in \partial(G)\setminus \sigma(G)$. By Lemma~\ref{lem:no_extreme}, there exist $a \in \partial(G)\setminus\{v\}$ such that $v$ is not a neighbor of $a$ in $G_{SR}$, so $G_{SR}$ is not a complete graph.
\end{proof}

If  $G$ is a connected graph  of order $n$, then $\sigma(G)=V(G)$ if and only if $G\cong K_n$. Hence, the following result is a direct consequence of  Theorem \ref{theorem:complete}.


\begin{corollary}
Let $G$ be a connected graph of order $n\ge 2$. Then $G_{SR}\cong K_n$ if and only if $G\cong K_n$.
\end{corollary}


Another particular case of Theorem \ref{theorem:complete} can be deduced from the next lemma. We recall that a {\em cut vertex} in a graph $G$ is a vertex when removed (together with its adjacent edges) from $G$ results in a new graph with increased number of connected components.

\begin{lemma}{\rm \cite{Rodriguez-Velazquez2016}}\label{LemmaCut}
Let $G$ be a connected graph. If $v$ is a cut vertex of $G$, then $v\not \in \partial(G)$.
\end{lemma}

\begin{proposition}
Let $G$ be a connected graph and let $\varepsilon(G)$ be the number of vertices of degree one. If every vertex of degree greater than one is a cut vertex of $G$, then $G_{SR}\cong K_{\varepsilon(G)}$.
\end{proposition}

In order to present  the next result we need to introduce some more terminology. Given a graph  $G$, we  define $G^*$ as the graph with vertex set $V(G^*)=V(G)$ such that  two vertices $u,v$ are adjacent in $G^*$ if and only if either $d_G(u,v)\ge 2$ or  $u,v$ are true twins.  If a graph $G$ has at least one isolated vertex, then we denote by $G_-$ the graph obtained from $G$ by removing all its isolated vertices. In this sense, $G^*_-$ is obtained from $G^*$ by removing all its isolated vertices.
Notice  that if $G$ is true twin-free, then $G^*\cong G^c$.

\begin{proposition}\label{proposition diameter two}
For any graph $G$ of diameter two, $G_{SR}\cong G^*_-$.
\end{proposition}

\begin{proof}
Assume that $G$ has diameter two and let $u,v$ be two different vertices of $G$. If $u\not\sim v$ or $N_G[u]=N_G[v]$, then $u$ and $v$ are MMD in $G$. Now, if $u\sim v$ and $N_G[u]\ne N_G[v]$, then there exists, $w\in V(G)\setminus \{u,v\}$ such that either ($w\sim u$ and $w\not\sim v$) or ($w\not \sim u$ and $w\sim v$), which implies that $u$ and $v$ are not MMD. Therefore, the result follows.
\end{proof}

\begin{theorem}{\rm \cite{EstrGarcRamRodr2015}}\label{GEqualGcSR}
Let $G$ be a connected graph. Then $G_{SR}\cong G^c$ if and only if  $D(G)=2$ and $G$ is a true twin-free graph.
\end{theorem}

\begin{proof}
Assume that $G_{SR}\cong G^c=(V,E)$, and let $u,v\in V$ be two diametral vertices in $G$. Since $u$ and $v$ are MMD in $G$ and $G_{SR}\cong G^c$, we obtain that $u$ and $v$ are adjacent in $G^c$ and, as a result, $D(G)=d_{G}(u,v)\ge 2$. Now, suppose that $d_{G}(u,v)>2$. Then  there exists $w\in N_{G}(v)-N_{G}(u)$ such that $d_{G}(u,w)=D(G)-1\ge 2$. Hence,  $w$ and $u$ are not MMD in $G$ and $w\in N_G(u)$, which contradicts the fact that $G_{SR}\cong G^c$. Therefore,  $D(G)=2$. Now assume that there exists two vertices $x$ and $y$ which are true twins in $G$. We have that $x$ and $y$ are false twins in $G^c$ and, as a result,  they are not adjacent in  $G^c$ and they are MMD in $G$, which contradicts  the fact that $G_{SR}\cong G^c$. Therefore, $G$ is a true twin-free graph.

On the other hand, if $G=(V,E)$ is a true twin-free graph and $D(G)=2$, then two vertices $u,v$ are MMD in $G$ if and only if $d_{G}(u,v)=2$. Therefore, $G_{SR}\cong G^c$.
\end{proof}

We next show that star graphs and complete bipartite graphs $K_{2,r}$ are not realizable as the strong resolving graph of any graph.

\begin{proposition}\label{stars-as-SRG}
Let $G$ be a connected graph of order $n\ge 2$ and let $r\ge 1$ be an integer. Then the following statement hold.
\begin{itemize}
\item $G_{SR}\cong K_{1,r}$ if and only if $G\cong P_n$ and $r=1$.
\item The graph equation $G_{SR}\cong K_{2,r}$ has no solution.
\end{itemize}
\end{proposition}

\begin{proof}

Obviously, $(P_n)_{SR}\cong K_2\cong K_{1,1}$. Now, if $G_{SR}\cong K_{1,r}$, then $\dim_s(G)=\beta(G_{SR})=1$, which implies that $G\cong P_n$, by Observation \ref{values-sdim-basic} (a), and so $r=1$. Therefore, the first statement holds.

Now, assume that $G_{SR}$ is a complete bipartite graph $(U_1\cup U_2,E)$, where $|U_1|,|U_2|\ge 2$. Since the subgraph of  $G_{SR}$ induced by $\sigma(G)$ is a clique, $|U_1\cap \sigma(G)|\le 1$ and $|U_2\cap \sigma(G)|\le 1$.
Hence, Lemma \ref{lem:no_extreme} immediately leads to $|U_1|\ge 3$ and $|U_2|\ge 3$, which implies that the graph equation $G_{SR}\cong K_{2,r}$ has no solution.
\end{proof}

It is worth mentioning that, concerning the result before, although no star graph $K_{1,r}$, $r\ge 2$, is a strong resolving graph, there are graphs $G$ for which $G_{SR}$ contains a component isomorphic to a star graph $K_{1,r}$ for any $r\ge 2$. To see this, consider the following family $\mathcal{F}$ of graphs $G_r$ constructed in the following way, that was already presented in \cite{Kang2016A}.
\begin{itemize}
  \item Consider $r+1$ paths $a_ib_ic_i$ with $i\in\{0,\ldots,r\}$.
  \item Add the edges $a_ia_0$, $b_ib_0$ and $c_ic_0$ for every $i\in\{1,\ldots,r\}$.
  \item Add a vertex $x$ and the edges $xa_0$ and $xc_0$.
\end{itemize}

An example of a graph in $\mathcal{F}$ and its strong resolving graph is given in Figure \ref{G_4}.

\begin{figure}[ht]
\centering
\begin{tikzpicture}[scale=.6, transform shape]
\node [draw, shape=circle] (b0) at  (0,0) {};
\node [draw, shape=circle] (c0) at  (4,0) {};
\node [draw, shape=circle] (c1) at  (3,2) {};
\node [draw, shape=circle] (c2) at  (3,4) {};
\node [draw, shape=circle] (c3) at  (5,3) {};
\node [draw, shape=circle] (c4) at  (5,5) {};
\node [draw, shape=circle] (b3) at  (1,3) {};
\node [draw, shape=circle] (b4) at  (1,5) {};
\node [draw, shape=circle] (a0) at  (-4,0) {};
\node [draw, shape=circle] (a1) at  (-5,2) {};
\node [draw, shape=circle] (a2) at  (-5,4) {};
\node [draw, shape=circle] (a3) at  (-3,3) {};
\node [draw, shape=circle] (a4) at  (-3,5) {};
\node [draw, shape=circle] (b1) at  (-1,2) {};
\node [draw, shape=circle] (b2) at  (-1,4) {};
\node [draw, shape=circle] (x) at  (0,-2) {};
\node [draw, shape=circle] (a11) at  (8.5,-2) {};
\node [draw, shape=circle] (a22) at  (8.5,0.5) {};
\node [draw, shape=circle] (a33) at  (8.5,3) {};
\node [draw, shape=circle] (a44) at  (8.5,5.5) {};
\node [draw, shape=circle] (c11) at  (11.5,-2) {};
\node [draw, shape=circle] (c22) at  (11.5,0.5) {};
\node [draw, shape=circle] (c33) at  (11.5,3) {};
\node [draw, shape=circle] (c44) at  (11.5,5.5) {};
\node [draw, shape=circle] (x1) at  (13.3,1.75) {};
\node [draw, shape=circle] (b11) at  (15,-2) {};
\node [draw, shape=circle] (b22) at  (15,0.5) {};
\node [draw, shape=circle] (b33) at  (15,3) {};
\node [draw, shape=circle] (b44) at  (15,5.5) {};
\node [scale=1.4] at (7.8,-2) {$a_1$};
\node [scale=1.4] at (7.8,0.5) {$a_2$};
\node [scale=1.4] at (7.8,3) {$a_3$};
\node [scale=1.4] at (7.8,5.5) {$a_4$};
\node [scale=1.4] at (12.2,-2) {$c_1$};
\node [scale=1.4] at (12.2,0.5) {$c_2$};
\node [scale=1.4] at (12.2,3) {$c_3$};
\node [scale=1.4] at (12.2,5.5) {$c_4$};
\node [scale=1.4] at (12.8,1.75) {$x$};
\node [scale=1.4] at (15.7,-2) {$b_1$};
\node [scale=1.4] at (15.7,0.5) {$b_2$};
\node [scale=1.4] at (15.7,3) {$b_3$};
\node [scale=1.4] at (15.7,5.5) {$b_4$};
\node [scale=1.4] at (0.7,-2.2) {$x$};
\node [scale=1.4] at (-4.2,-0.7) {$a_0$};
\node [scale=1.4] at (0,-0.7) {$b_0$};
\node [scale=1.4] at (4.2,-0.7) {$c_0$};
\node [scale=1.4] at (-5.2,1.3) {$a_1$};
\node [scale=1.4] at (-1.2,1.3) {$b_1$};
\node [scale=1.4] at (2.8,1.3) {$c_1$};
\node [scale=1.4] at (-2.5,2.5) {$a_4$};
\node [scale=1.4] at (1.5,2.5) {$b_4$};
\node [scale=1.4] at (5.5,2.5) {$c_4$};
\node [scale=1.4] at (-5,4.55) {$a_2$};
\node [scale=1.4] at (-1,4.55) {$b_2$};
\node [scale=1.4] at (3,4.55) {$c_2$};
\node [scale=1.4] at (-3,5.55) {$a_3$};
\node [scale=1.4] at (1,5.55) {$b_3$};
\node [scale=1.4] at (5,5.55) {$c_3$};
\node [scale=1.4] at (0,-3.5) {\large $G_4$};
\node [scale=1.4] at (12.4,-3.5) {\large $(G_4)_{SR}$};
\draw(x)--(a0)--(a4)--(b4)--(c4)--(c0)--(x);
\draw(b0)--(a0)--(a3)--(b3)--(c3)--(c0)--(b0);
\draw(a0)--(a2)--(b2)--(c2)--(c0);
\draw(a0)--(a1)--(b1)--(c1)--(c0);
\draw(b1)--(b0)--(b2);
\draw(b3)--(b0)--(b4);
\draw(a11)--(c22)--(a33)--(c44)--(a11)--(c33)--(a44)--(c22);
\draw(a44)--(c11)--(a22)--(c33);
\draw(a22)--(c44);
\draw(c11)--(a33);
\draw(b33)--(x1)--(b44);
\draw(b11)--(x1)--(b22);
\end{tikzpicture}
\caption{The graph $G_4\in \mathcal{F}$ and its strong resolving graph.}\label{G_4}
\end{figure}
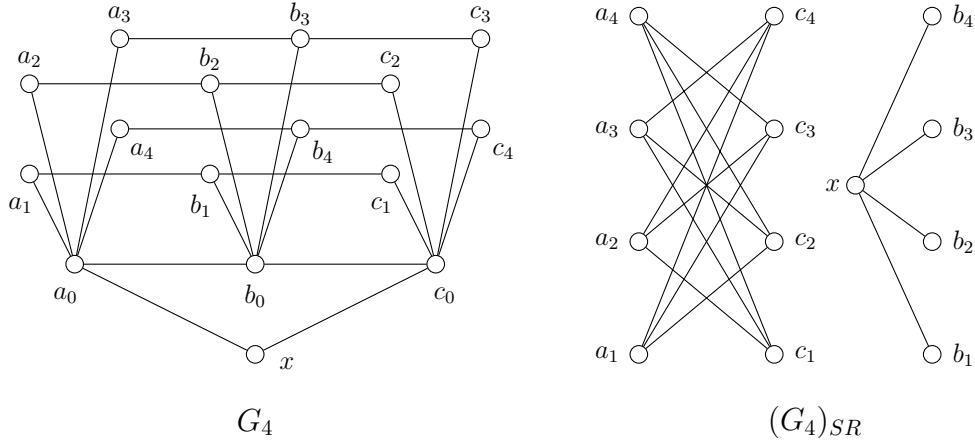
We can easily notice the following.
\begin{itemize}
  \item The vertex $a_i$ is MMD only with the vertices $c_j$ such that $j\ne 0,i$.
  \item Similarly, the vertex $c_i$ is MMD only with the vertices $a_j$ such that $j\ne 0,i$.
  \item The vertex $b_i$ is MMD only with the vertex $x$ and viceversa.
  \item The vertices $a_0,b_0,c_0$ are not MMD with any vertex in $G_r$.
\end{itemize}
As a consequence of the facts above it clearly happens that $(G_r)_{SR}$ contains two connected components. One of them isomorphic to a star graph $S_{1,r}$ with $r$ leaves, and the second one isomorphic to a complete bipartite graph $K_{r,r}$ minus a perfect matching.

Other non realization result for strong resolving graphs comes whether we consider the cycle $C_4$ as a possible strong resolving graph.

%

By Proposition \ref{stars-as-SRG}  we learned that the graph equations  $G_{SR}\cong K_{1,r}$ and  $G_{SR}\cong K_{2,r}$, for $r\ge 2$, have no solution. We propose the following conjecture.

\begin{conjecture}
The graph equation  $G_{SR}\cong K_{r,s}$ has no solution for any $r,s\ge 2$.
\end{conjecture}

Our next result concerns the equation $G_{SR}\cong P_n$, with $n\ne 3$. To this end, we consider the family $\mathcal{F}_P$ of graphs $G_P^n$ with $n\ge 5$ given as follows.
\begin{itemize}
  \item We begin with a path on $n-1$ vertices $v_1v_2\ldots v_{n-1}$.
  \item If $n$ is even, then
  \begin{itemize}
    \item add $\frac{n-2}{2}$ vertices $a_1,a_2,\ldots, a_{(n-2)/2}$ and $\frac{n-2}{2}$ vertices $b_1,b_2,\ldots, b_{(n-2)/2}$,
    \item add the edges $a_iv_{2i-1}$, $a_iv_{2i+1}$ with $i\in\{1,\ldots,(n-2)/2\}$, the edges $b_iv_{2i}$, $b_iv_{2i+2}$ with $i\in\{1,\ldots,(n-4)/2\}$ and the edges $b_{(n-2)/2}v_{n-2}$, $b_{(n-2)/2}v_{n-1}$.
  \end{itemize}
  \item If $n$ is odd, then
  \begin{itemize}
    \item add $\frac{n-1}{2}$ vertices $a_1,a_2,\ldots, a_{(n-1)/2}$ and $\frac{n-3}{2}$ vertices $b_1,b_2,\ldots, b_{(n-3)/2}$,
    \item add the edges $a_iv_{2i-1}$, $a_iv_{2i+1}$ with $i\in\{1,\ldots,(n-3)/2\}$, the edges $a_{(n-1)/2}v_{n-2}$, $a_{(n-1)/2}v_{n-1}$ and the edges $b_iv_{2i}$, $b_iv_{2i+2}$ with $i\in\{1,\ldots,(n-3)/2\}$.
  \end{itemize}
\end{itemize}
Examples of the graphs of family $\mathcal{F}_P$ are given in Figure \ref{G_P-Ex}.

\begin{figure}[ht]
\centering
\begin{tikzpicture}[scale=.6, transform shape]
\node [draw, shape=circle] (v1) at  (0,4) {};
\node [draw, shape=circle] (v2) at  (2,4) {};
\node [draw, shape=circle] (v3) at  (4,4) {};
\node [draw, shape=circle] (v4) at  (6,4) {};
\node [draw, shape=circle] (v5) at  (8,4) {};
\node [draw, shape=circle] (v6) at  (10,4) {};
\node [draw, shape=circle] (v7) at  (12,4) {};
\node [draw, shape=circle] (v8) at  (14,4) {};

\node [draw, shape=circle] (a1) at  (2,6) {};
\node [draw, shape=circle] (a2) at  (6,6) {};
\node [draw, shape=circle] (a3) at  (10,6) {};
\node [draw, shape=circle] (a4) at  (14,6) {};

\node [draw, shape=circle] (b1) at  (4,2) {};
\node [draw, shape=circle] (b2) at  (8,2) {};
\node [draw, shape=circle] (b3) at  (12,2) {};

\node [draw, shape=circle] (v11) at  (0,0.1) {};
\node [draw, shape=circle] (v21) at  (2,0.1) {};
\node [draw, shape=circle] (v31) at  (4,0.1) {};
\node [draw, shape=circle] (v41) at  (6,0.1) {};
\node [draw, shape=circle] (v51) at  (8,0.1) {};
\node [draw, shape=circle] (v61) at  (10,0.1) {};
\node [draw, shape=circle] (v71) at  (12,0.1) {};
\node [draw, shape=circle] (v81) at  (14,0.1) {};
\node [draw, shape=circle] (v91) at  (16,0.1) {};

\node [scale=1.4] at (2,6.5) {$a_1$};
\node [scale=1.4] at (6,6.5) {$a_2$};
\node [scale=1.4] at (10,6.5) {$a_3$};
\node [scale=1.4] at (14,6.5) {$a_4$};
\node [scale=1.4] at (4,1.4) {$b_1$};
\node [scale=1.4] at (8,1.4) {$b_2$};
\node [scale=1.4] at (12,1.4) {$b_3$};

\node [scale=1.4] at (2,4.5) {$v_2$};
\node [scale=1.4] at (6,4.5) {$v_4$};
\node [scale=1.4] at (10,4.5) {$v_6$};
\node [scale=1.4] at (14.6,4) {$v_8$};
\node [scale=1.4] at (-0.6,4) {$v_1$};
\node [scale=1.4] at (4,3.5) {$v_3$};
\node [scale=1.4] at (8,3.5) {$v_5$};
\node [scale=1.4] at (12,3.5) {$v_7$};

\node [scale=1.4] at (2,-0.5) {$b_1$};
\node [scale=1.4] at (6,-0.5) {$b_2$};
\node [scale=1.4] at (10,-0.5) {$b_3$};
\node [scale=1.4] at (0,-0.5) {$a_1$};
\node [scale=1.4] at (4,-0.5) {$a_2$};
\node [scale=1.4] at (8,-0.5) {$a_3$};
\node [scale=1.4] at (12,-0.5) {$a_4$};
\node [scale=1.4] at (14,-0.5) {$v_1$};
\node [scale=1.4] at (16,-0.5) {$v_8$};

\node [scale=1.4] at (-2.5,4) {\large $G_P^{9}$};
\node [scale=1.4] at (-2.5,0) {\large $(G_P^{9})_{SR}$};

\draw(v1)--(a1)--(v3)--(a2)--(v5)--(a3)--(v7)--(a4)--(v8)--(v7)--(v6)--(v5)--(v4)--(v3)--(v2)--(v1);
\draw(v2)--(b1)--(v4)--(b2)--(v6)--(b3)--(v8);

\draw(v91)--(v81)--(v71)--(v61)--(v51)--(v41)--(v31)--(v21)--(v11);
\end{tikzpicture}

\begin{tikzpicture}[scale=.5, transform shape]
\node [draw, shape=circle] (v1) at  (0,4) {};
\node [draw, shape=circle] (v2) at  (2,4) {};
\node [draw, shape=circle] (v3) at  (4,4) {};
\node [draw, shape=circle] (v4) at  (6,4) {};
\node [draw, shape=circle] (v5) at  (8,4) {};
\node [draw, shape=circle] (v6) at  (10,4) {};
\node [draw, shape=circle] (v7) at  (12,4) {};
\node [draw, shape=circle] (v8) at  (14,4) {};
\node [draw, shape=circle] (v9) at  (16,4) {};

\node [draw, shape=circle] (a1) at  (2,6) {};
\node [draw, shape=circle] (a2) at  (6,6) {};
\node [draw, shape=circle] (a3) at  (10,6) {};
\node [draw, shape=circle] (a4) at  (14,6) {};

\node [draw, shape=circle] (b1) at  (4,2) {};
\node [draw, shape=circle] (b2) at  (8,2) {};
\node [draw, shape=circle] (b3) at  (12,2) {};
\node [draw, shape=circle] (b4) at  (16,2) {};

\node [draw, shape=circle] (v11) at  (0,0.1) {};
\node [draw, shape=circle] (v21) at  (2,0.1) {};
\node [draw, shape=circle] (v31) at  (4,0.1) {};
\node [draw, shape=circle] (v41) at  (6,0.1) {};
\node [draw, shape=circle] (v51) at  (8,0.1) {};
\node [draw, shape=circle] (v61) at  (10,0.1) {};
\node [draw, shape=circle] (v71) at  (12,0.1) {};
\node [draw, shape=circle] (v81) at  (14,0.1) {};
\node [draw, shape=circle] (v91) at  (16,0.1) {};
\node [draw, shape=circle] (v101) at  (18,0.1) {};

\node [scale=1.4] at (2,6.5) {$a_1$};
\node [scale=1.4] at (6,6.5) {$a_2$};
\node [scale=1.4] at (10,6.5) {$a_3$};
\node [scale=1.4] at (14,6.5) {$a_4$};
\node [scale=1.4] at (4,1.4) {$b_1$};
\node [scale=1.4] at (8,1.4) {$b_2$};
\node [scale=1.4] at (12,1.4) {$b_3$};
\node [scale=1.4] at (16,1.4) {$b_4$};

\node [scale=1.4] at (2,4.5) {$v_2$};
\node [scale=1.4] at (6,4.5) {$v_4$};
\node [scale=1.4] at (10,4.5) {$v_6$};
\node [scale=1.4] at (14,4.5) {$v_8$};
\node [scale=1.4] at (-0.6,4) {$v_1$};
\node [scale=1.4] at (4,3.5) {$v_3$};
\node [scale=1.4] at (8,3.5) {$v_5$};
\node [scale=1.4] at (12,3.5) {$v_7$};
\node [scale=1.4] at (16.6,4) {$v_9$};

\node [scale=1.4] at (2,-0.5) {$b_1$};
\node [scale=1.4] at (6,-0.5) {$b_2$};
\node [scale=1.4] at (10,-0.5) {$b_3$};
\node [scale=1.4] at (14,-0.5) {$b_4$};
\node [scale=1.4] at (0,-0.5) {$a_1$};
\node [scale=1.4] at (4,-0.5) {$a_2$};
\node [scale=1.4] at (8,-0.5) {$a_3$};
\node [scale=1.4] at (12,-0.5) {$a_4$};
\node [scale=1.4] at (16,-0.5) {$v_1$};
\node [scale=1.4] at (18,-0.5) {$v_9$};

\node [scale=1.4] at (-2.5,4) {\large $G_P^{10}$};
\node [scale=1.4] at (-2.5,0) {\large $(G_P^{10})_{SR}$};

\draw(v1)--(a1)--(v3)--(a2)--(v5)--(a3)--(v7)--(a4)--(v9)--(v8)--(v7)--(v6)--(v5)--(v4)--(v3)--(v2)--(v1);
\draw(v2)--(b1)--(v4)--(b2)--(v6)--(b3)--(v8)--(b4)--(v9);

\draw(v101)--(v91)--(v81)--(v71)--(v61)--(v51)--(v41)--(v31)--(v21)--(v11);
\end{tikzpicture}
\caption{The graphs $G_P^{9}$ and $G_P^{10}$ in $\mathcal{F}_P$ and their strong resolving graphs.}\label{G_P-Ex}
\end{figure}
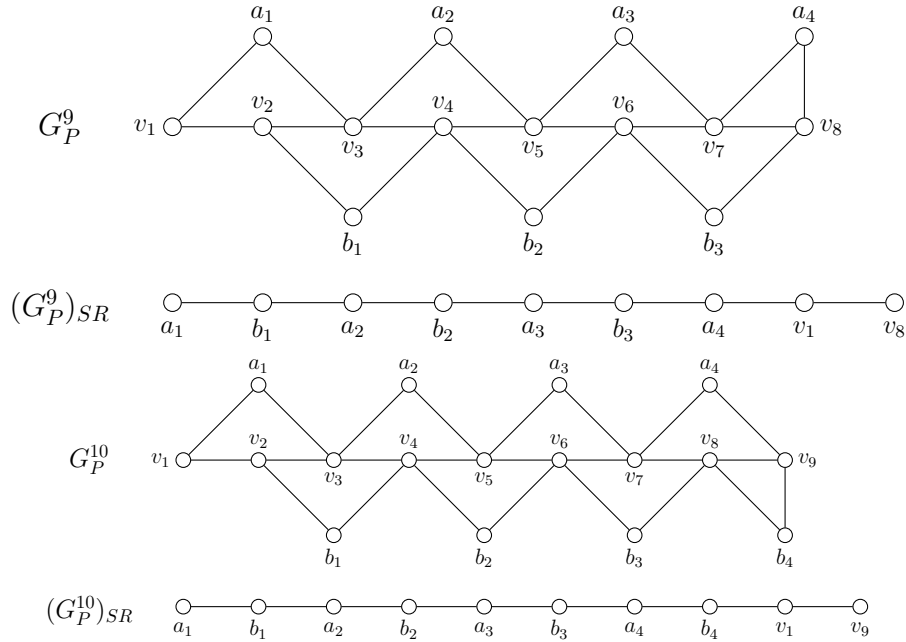

\begin{proposition}
For any integer $n\ge 2$ and $n\ne 3$, there exists a graph $G$ such that $G_{SR}\cong P_n$.
\end{proposition}

\begin{proof}
  If $n=2$, then any path $P_t$ satisfies that $(P_t)_{SR}\cong P_2$. If $n=3$, then by Proposition \ref{stars-as-SRG} we know there is no graph $G$ such $G_{SR}\cong G$. If $n=4$, then consider the join graph $K_1+P_4$, for which it is not difficult to see that $(K_1+P_4)_{SR}\cong P_4$. If $n\ge 5$, then we consider a graph $G_P^n\in \mathcal{F}_P$, where the following facts are observed. Assume $n$ is even.
  \begin{itemize}
    \item Every vertex $a_i$, with $i\in \{2,\dots,(n-2)/2\}$, is only MMD with the vertices $b_i$ and $b_{i-1}$.
    \item The vertex $a_1$ is only MMD with the vertex $b_1$.
    \item Every vertex $b_i$, with $i\in \{1,\dots,(n-4)/2\}$, is only MMD with the vertices $a_i$ and $a_{i+1}$.
    \item The vertex $b_{(n-2)/2}$ is only MMD with the vertices $a_{(n-2)/2}$ and $v_1$.
    \item The vertices $v_1$ and $v_{n-1}$ are MMD between them.
    \item No vertex $v_i$,  with $i\in \{2,\dots,n-2\}$, belongs to the boundary of $G_P^n$.
  \end{itemize}
  According to the items above it clearly follows that $(G_P^n)_{SR}$ is isomorphic to the path $P_n=a_1b_1a_2b_2\cdots a_{(n-2)/2}b_{(n-2)/2}v_1v_{n-1}$. A similar procedure can be used for the case $n$ odd, which completes the proof.
\end{proof}

Our next   result concerns the realization  of cycles $C_n$ as strong resolving graphs. From Observation \ref{observation1} (iii) we know that for any odd cycle $C_{2k+1}$, it follows $(C_{2k+1})_{SR}\cong C_{2k+1}$. Also, from Proposition \ref{stars-as-SRG}, the cycle $C_4\cong K_{2,2}$ is not realizable as the strong resolving graph of any graph. In general, the following can be stated.

\begin{proposition}
For any integer $n\ge 3$ and $n\ne 4$, there exists a graph $G$ such that $G_{SR}\cong C_n$.
\end{proposition}

\begin{proof}
  If $n=3$, then clearly $(C_3)_{SR}\cong C_3$. Consider a cycle graph of order $n\ge 5$. Since $C_n^c$ is a twin-free graph and has diameter two, by Theorem \ref{GEqualGcSR}, $(C_n^c)_{SR}\cong (C_n^c)^c\cong C_n$. That is, the complement of a cycle of order $n$ gives a strong resolving graph isomorphic to the cycle $C_n$, which completes the realization.
\end{proof}

More in general, since $D(G)\ge 4$ leads to $D(G^c)=2$, the following result is a direct consequence of Theorem \ref{GEqualGcSR}.

\begin{corollary}
Any false twin-free graph of diameter greater than or equal to four is the strong resolving graph of a true twin-free graph of diameter two.
\end{corollary}

A summary of the results we obtained related to the Realization Theorem with complete and complete bipartite graphs, paths and cycles can be found in Table~\ref{table:realization}.

\def\arraystretch{1.8}
\begin{table}[h]\label{table:realization}
\centering
\caption{Notable graphs families and Realization Problem}
\label{my-label}
\begin{tabular}{l|l|l|l|l|l|l}
 & $K_n\ (n\geq 2)$ &  $K_{1,r}\ (r\geq 1)$ &  $K_{2,r}\ (r\geq 1)$ & $K_{s,r}\ (r\geq 3)$ & $P_n\ (n\geq 2)$ & $C_n\ (n\geq 3)$ \\
 \hline
$=G_{SR}$  & $n\geq 2$ & $r=1$ & none & unknown & $n\neq 3$  & $n\neq 4$\\
 \hline
$\neq G_{SR}$ & none & $r\geq 2$ & $r\geq 1$ & unknown & $n=3$& $n=4$
\end{tabular}
\end{table}

\section{Strong Resolving Graph of Product Graphs}\label{SectionCharacterization problem for product graphs}

We begin this section with a brief overview on products of graphs, of those ones which will be further considered. According to the two books \cite{Hammack2011,Imrich2000}, a graph product of the graphs $G$ and $H$ means a graph whose vertex set is defined on the cartesian product $V(G)\times V(H)$ of the vertex sets of $G$ and $H$, and edges are determined by a function on the edges of $G$ and $H$. The graphs $G$ and $H$ are called the {\em factor graphs}. Considering such mentioned rules, there are exactly 256 possible products. However, according to several their properties such as associativity, commutativity, complementarity, etc., the most common and well investigated are the Cartesian product, the direct product, the strong product, and the lexicographic product, which are also known as the \emph{standard products} \cite{Hammack2011,Imrich2000}. Nonetheless, there exist other less known operations with graphs which are interesting for some investigations, for instance we could mention the Cartesian sum graph and the corona product graphs, among other ones.

Studies on finding relationships between properties of product graphs and properties of the factors have attracted several researchers in the last recent year. The case of strong metric generators has not escaped to this and several investigations have been published concerning this. In such researches a powerful tool has been deducing the structure of the strong resolving graph of a product from that of its factors. In this section we survey precisely some results concerning the strong resolving graphs of product graphs, but we previously gives some background on their definitions and basic properties.

The {\em direct product} of two graphs $G$ and $H$ is the graph $G\times H$, such that $V(G\times H)=V(G)\times V(H)$ and two vertices $(a,b),(c,d)$ are adjacent in $G\times H$ if and only if
\begin{itemize}
\item $ac\in E(G)$ and
\item $bd\in E(H)$.
\end{itemize}

The direct product is also known as the \emph{Kronecker product}, the \emph{tensor product}, the \emph{categorical product}, the \emph{cardinal product}, the \emph{cross product}, the \emph{conjunction}, the \emph{relational product} or the \emph{weak direct product}. This product is commutative and associative in a natural way \cite{Hammack2011,Imrich2000}. The distance and connectedness in the direct product are more subtle than for other products. The formula on the vertex distances in the direct product is the following.
\begin{remark}{\em \cite{Kim1991}}\label{dir-distance}
For any graphs $G$ and $H$ and any two vertices $(a,b)$, $(c,d)$ of $G\times H$,
$$d_{G\times H}((a,b),(c,d))=\min \{\max \{d_{G}^{e}(a,c),d_{H}^{e}(b,d)\}, \max \{d_{G}^{o}(a,c),d_{H}^{o}(b,d)\}\},$$
where $d_{G}^{e}(a,c)$ means the length of a shortest walk of even
length between $a$ and $c$ in $G$, and $d_{G}^{o}(a,c)$ the length of a shortest odd walk between $a$ and $c$ in $G$. If such a walk does not exist, we set $d_{G}^{e}(a,c)$ or $d_{G}^{o}(a,c)$ to be infinite.
\end{remark}

On the other hand, the connectedness in the direct product of two graphs relies on the bipartite properties of the factor graphs, namely, the result presented at next.

\begin{theorem}{\em \cite{Weichsel1962}}\label{dir-connected}
A direct product of nontrivial graphs is connected if and only if both factors are connected and at least one factor is nonbipartite.
\end{theorem}

In contrast to distances, the direct product is the most natural product for open neighborhoods:
\begin{equation}\label{neigh}
N_{G\times H}(a,b)=N_G(a)\times N_H(b).
\end{equation}


The \textit{Cartesian product} of two graphs $G$ and $H$ is the graph $G\Box H$, such that $V(G\Box H)=V(G)\times V(H)$ and two vertices $(a,b),(c,d)\in V(G\Box H)$ are adjacent in $G\Box H$ if and only if either
\begin{itemize}
\item $a=c$ and $bd\in E(H)$, or
\item $ac\in E(G)$  and $b=d$.
\end{itemize}

The Cartesian product is a straightforward and natural construction, and is in many respects the simplest graph product \cite{Hammack2011,Imrich2000}. Hypercubes, Hamming graphs and grid graphs are some particular cases of this product. The \emph{Hamming graph} $H_{k,n}$ is the Cartesian product of $k$ copies of the complete graph $K_n$, \emph{i.e.},
\[\begin{array}{c}
            H_{k,n}=\underbrace{K_n\;\Box\; K_n\;\Box\; ...\; \Box\; K_n} \\
            \;\;\;\;\;\;\;\;\;\;k \mbox{ times}
          \end{array}
\]

The \emph{hypercube} $Q_n$ is defined as $H_{n,2}$. Moreover, the \emph{grid graph} $P_k\Box P_n$ is the Cartesian product of the paths $P_k$ and $P_n$, the \emph{cylinder graph} $C_k\Box P_n$ is the Cartesian product of the cycle $C_k$ and the path $P_n$, and the \emph{torus graph} $C_k\Box C_n$ is the Cartesian product of the cycles $C_k$ and $C_n$.

The Cartesian product is a commutative and associative operation. Moreover, it is connected whenever the factors are both connected. The distance between any two of its vertices is given by
$$d_{G\Box H}((a,b),(c,d))=d_G(a,c)+d_H(b,d)$$
while the neighborhood of a vertex $(a,b)\in V(G\Box H)$ is
\begin{equation*}
N_{G\Box H}(a,b)=(N_G(a)\times \{b\})\cup (\{a\}\times N_H(b)).
\end{equation*}

The \textit{strong product} of two graphs $G$ and $H$ is the graph $G\boxtimes H$ such that $V(G\boxtimes H)=V(G)\times V(H)$, and two vertices $(a,b),(c,d)\in V(G\boxtimes H)$ are adjacent in $G\boxtimes H$ if and only if either
\begin{itemize}
\item $a=c$ and $bd\in E_2$, or
\item $ac\in E_1$ and $b=d$, or
\item $ac\in E_1$ and $bd\in E_2$.
\end{itemize}

Similarly to the Cartesian product, the strong product is a commutative and associative operation and, it is connected whenever the factors are both connected. The distance between any two of its vertices is computed by using the following formula
$$d_{G\boxtimes H}((a,b),(c,d))=\max\{d_G(a,c),d_H(b,d)\}.$$
On the other hand, the neighborhood of a vertex $(a,b)\in V(G\boxtimes H)$ is given by
\begin{equation*}
N_{G\boxtimes H}(a,b)=N_G[a]\times N_H[b].
\end{equation*}


The \textit{lexicographic product} of two graphs $G$ and $H$ is the graph $G\circ H$ with vertex set $V(G\circ H)=V(G)\times V(H)$ and two vertices $(a,b)\in V(G\circ H)$ and $(c,d)\in V(G\circ H)$ are adjacent in $G\circ H$ if and only if either
\begin{itemize}
\item $ac\in E_1$, or
\item $a=c$ and $bd\in E_2$.
\end{itemize}

Note that the lexicographic product of two graphs is the only not commutative operation among the four standard products. Moreover, $G\circ H$ is a connected graph if and only if $G$ is connected. The distances and neighborhoods in the lexicographic product are obtained as the following known results show.

\begin{theorem}{\rm \cite{Hammack2011}}\label{basictoolLexicographic} Let $G$ and $H$ be two nontrivial graphs such that $G$ is connected. Then the following assertions hold for any  $a,c\in V(G)$ and $b,d\in V(H)$ such  that $a\ne c$.
\begin{enumerate}[{\rm (i)}]
\item $N_{G\circ H}(a,b)=\left(\{a\}\times  N_H(b)\right)\cup \left( N_G(a)\times  V(H)\right)$.
\item  $d_{G\circ H}((a,b),(c,d)) = d_{G}(a,c)$
\item   $d_{G\circ H}((a,b),(a,d)) = \min \{d_{H}(b,d),2\}$.
\end{enumerate}
\end{theorem}


The \textit{Cartesian sum} of two graphs $G$ and $H$, denoted by $G\oplus H$, is the graph with vertex set $V(G\circ H)=V(G)\times V(H)$, where $(a,b)(c,d)\in E(G\oplus H)$ if and only if
\begin{itemize}
\item $ac\in E_1$, or
\item $bd\in E_2$.
\end{itemize}

This notion of graph product was introduced by Ore \cite{Ore1962} in 1962, nevertheless it has passed almost unnoticed and just few results (for instance \cite{Cizek1994,Scheinerman1997}) have been presented about this. The Cartesian sum is also known as the \emph{disjunctive product} \cite{Scheinerman1997} and it is a commutative and associative operation \cite{Hammack2011}.



Next result summarizes some properties about the diameter of the Cartesian sum graph.

\begin{proposition}{\rm \cite{Kuziak2014b}}\label{lem Cart sum diam}
Let $G$ and $H$ be two nontrivial graphs such that at least one of them is noncomplete and let $n\ge 2$ be an integer. Then the following assertions hold.
\begin{enumerate}[{\rm (i)}]
\item $D(G\oplus N_n)=\max\{2,D(G)\}.$
\item If $G$ and $H$ have isolated vertices, then  $D(G\oplus H)=\infty$.
\item If neither $G$ nor  $H$ has isolated vertices, then $D(G\oplus H)=2$.
\item If  $D(H)\le 2$, then $D(G\oplus H)=2$.
\item If $D(H)>2$, $H$ has no isolated vertices and $G$ is a nonempty graph having at least one isolated vertex, then $D(G\oplus H)=3$.
\end{enumerate}
\end{proposition}

The neighborhood of a vertex $(a,b)\in V(G\oplus H)$ is
\begin{equation*}
N_{G\oplus H}(a,b)=(N_G(a)\times V(H)) \cup (V(G)\times N_H(b)).
\end{equation*}

The \emph{corona product} $G\odot H$ is defined as the graph obtained from $G$ and $H$ by taking one copy of $G$ and $n=|V(G)|$ copies of $H$ and joining by an edge each vertex from the $i^{th}$-copy of $H$ with the $i^{th}$-vertex of $G$. We denote by $V=\{v_1,v_2,\dots ,v_n\}$ the set of vertices of $G$ and by $H_i=(V_i,E_i)$ the copy of $H$ such that $v_i\sim v$ for every $v\in V_i$. Observe that $G\odot H$ is connected if and only if $G$ is connected. The concept of corona product of two graphs was first introduced by Frucht and Harary \cite{Frucht1970}.

The following expression for the distance between two vertices $x,y$ of  $G\odot H$ is a direct consequence of the definition of corona product graph.
\begin{equation}\label{DistanceCoronaGraphs}
d_{G\odot H}(x,y)=\left\lbrace
\begin{array}{ll}
d_G(x,y),& x,y\in V;\\ \\
d_G(v_i,v_j)+1,& x=v_i, \, y\in V_j;\\ \\
d_G(v_i,v_j)+2,& x\in V_i, \, y\in V_j,\, i\ne j;\\ \\
\min\{d_{H_i}(x,y),2\}, & x,y\in V_i.
\end{array}
\right.
\end{equation}

\subsection{Cartesian Product and Direct Product of Graphs}

The next result establishes an interesting connection between the strong resolving graph of the Cartesian product of two graphs and the direct product of the strong resolving graphs of its factors. Such result was a powerful tool used in  \cite{RodriguezVelazquez2014a} while studying the strong metric dimension of Cartesian product graphs.

\begin{theorem}{\rm \cite{RodriguezVelazquez2014a}}\label{cartesian_directed}
Let $G$ and $H$ be two connected graphs. Then
$$(G\Box H)_{SR}\cong G_{SR}\times H_{SR}.$$
\end{theorem}

\begin{proof} Let $(g,h), (g',h')$ be any two vertices of $G \Box H$. Then, we have
$$d_{G\Box H}((g,h), (g',h'))= d_G(g,g')+d_H(h,h').$$ Thus, if $g''\sim g'$ and $d_G(g,g'')=d_G(g,g')+1$, then $(g',h') \sim (g'',h')$ and $d_{G \Box H}((g,h),(g'',h'))= d_G(g,g') + d_H(h,h') +1=d_{G\Box H}((g,h), (g',h'))+1$.

Using these observations, it is readily seen that $(g,h)$ and $(g',h')$ are MMD if and only if $g$ and $g'$ are MMD in $G$ and $h$ and $h'$ are MMD in $H$.
Moreover, $(g,h)(g',h') \in E((G\Box H)_{SR})$ if and only if $gg' \in E(G_{SR})$ and $hh' \in E(H_{SR})$. Thus
 \[V((G\Box H)_{SR})=\partial (G\Box H)=\partial(G)\times \partial(H)=V(G_{SR}\times H_{SR}),\]
and \[(G\Box H)_{SR}\cong G_{SR}\times H_{SR}.\]
\end{proof}

Figure \ref{exch Cartesian} illustrates Cartesian product of two cycles of order three and its strong resolving graph. Since the strong resolving graph of $C_3$  is isomorphic to $C_3$, we can easily observe that $(C_3\Box C_3)_{SR}$ is isomorphic to $(C_3)_{SR}\times (C_3)_{SR}$.

\begin{figure}[ht]
\centering
\begin{tabular}{cccccc}
\begin{tikzpicture}

\draw(0,0)--(3,0);
\draw(0,1.5)--(3,1.5);
\draw(0,3)--(3,3);
\draw(0,0)--(0,3);
\draw(1.5,0)--(1.5,3);
\draw(3,0)--(3,3);

\draw (0,0) .. controls (-0.4,1) and (-0.4,2) .. (0,3);
\draw (1.5,0) .. controls (1.1,1) and (1.1,2) .. (1.5,3);
\draw (3,0) .. controls (2.6,1) and (2.6,2) .. (3,3);
\draw (0,0) .. controls (1,0.4) and (2,0.4) .. (3,0);
\draw (0,1.5) .. controls (1,1.9) and (2,1.9) .. (3,1.5);
\draw (0,3) .. controls (1,3.4) and (2,3.4) .. (3,3);

\filldraw[fill opacity=0.9,fill=white]  (0,0) circle (0.12cm);
\filldraw[fill opacity=0.9,fill=white]  (0,1.5) circle (0.12cm);
\filldraw[fill opacity=0.9,fill=white]  (0,3) circle (0.12cm);
\filldraw[fill opacity=0.9,fill=white]  (1.5,0) circle (0.12cm);
\filldraw[fill opacity=0.9,fill=white]  (1.5,1.5) circle (0.12cm);
\filldraw[fill opacity=0.9,fill=white]  (1.5,3) circle (0.12cm);
\filldraw[fill opacity=0.9,fill=white]  (3,0) circle (0.12cm);
\filldraw[fill opacity=0.9,fill=white]  (3,1.5) circle (0.12cm);
\filldraw[fill opacity=0.9,fill=white]  (3,3) circle (0.12cm);

\node [below left] at (0,0) {$a1$};
\node [right] at (0,1.2) {$a2$};
\node [above] at (-0.1,3.1) {$a3$};
\node [below] at (1.5,-0.1) {$b1$};
\node [right] at (1.5,1.2) {$b2$};
\node [right] at (1.5,2.7) {$b3$};
\node [below right] at (3,0) {$c1$};
\node [right] at (3,1.3) {$c2$};
\node [above right] at (3,3) {$c3$};

\end{tikzpicture} & & \hspace*{0.7cm} & &
\begin{tikzpicture}
\draw(0,0)--(3,3);
\draw(3,0)--(0,3);
\draw(1.5,0)--(3,1.5)--(1.5,3)--(0,1.5)--cycle;
\draw(0,0)--(1.5,3);
\draw(1.5,0)--(0,3);
\draw(1.5,0)--(3,3);
\draw(3,0)--(1.5,3);
\draw(0,0)--(3,1.5);
\draw(0,1.5)--(3,0);
\draw(0,1.5)--(3,3);
\draw(0,3)--(3,1.5);

\draw (0,0) .. controls (0.5,1) and (2,2.5) .. (3,3);
\draw (0,3) .. controls (1,2.5) and (2.5,1) .. (3,0);

\filldraw[fill opacity=0.9,fill=white]  (0,0) circle (0.12cm);
\filldraw[fill opacity=0.9,fill=white]  (0,1.5) circle (0.12cm);
\filldraw[fill opacity=0.9,fill=white]  (0,3) circle (0.12cm);
\filldraw[fill opacity=0.9,fill=white]  (1.5,0) circle (0.12cm);
\filldraw[fill opacity=0.9,fill=white]  (1.5,1.5) circle (0.12cm);
\filldraw[fill opacity=0.9,fill=white]  (1.5,3) circle (0.12cm);
\filldraw[fill opacity=0.9,fill=white]  (3,0) circle (0.12cm);
\filldraw[fill opacity=0.9,fill=white]  (3,1.5) circle (0.12cm);
\filldraw[fill opacity=0.9,fill=white]  (3,3) circle (0.12cm);

\node [below left] at (0,0) {$a1$};
\node [left] at (-0.1,1.5) {$a2$};
\node [above left] at (0,3) {$a3$};
\node [below] at (1.5,-0.1) {$b1$};
\node [below] at (1.5,1.4) {$b2$};
\node [above] at (1.5,3.1) {$b3$};
\node [below right] at (3,0) {$c1$};
\node [right] at (3.1,1.5) {$c2$};
\node [above right] at (3,3) {$c3$};

\end{tikzpicture} \\
\end{tabular}
\caption{Cartesian product graph $C_3\Box C_3$ and its strong resolving graph $(C_3\Box C_3)_{SR}$.}
\label{exch Cartesian}
\end{figure}
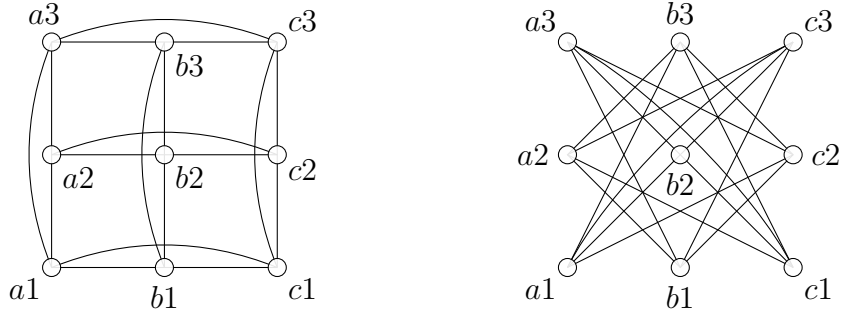

A \emph{matching} on a graph $G$ is a set of edges of $G$ such that no two edges share a vertex in common. A matching is \emph{maximum} if it has the maximum possible cardinality. Moreover, if every vertex of the graph is incident to exactly one edge of the matching, then it is called a \emph{perfect matching}.

The next result, implicitly deduced in  \cite[Proof of Theorem 6]{RodriguezVelazquez2014a},   deals with graphs whose strong resolving graphs are bipartite with a perfect matching.

\begin{theorem}\label{theoremMatching}
Let $G$ and $H$ be two connected graphs such that $H_{SR}$ is bipartite with a perfect matching.
Let $G_i$, $i\in \{1,\dots ,k\}$, be the connected components of $G_{SR}$. If for each $i\in \{1,\dots ,k\}$, $G_i$ is Hamiltonian or $G_i$ has a perfect matching, then $(G\Box H)_{SR}$
 is bipartite and has a perfect matching.
\end{theorem}

\begin{proof}
Since $H_{SR}$ is bipartite, $G_{SR}\times H_{SR}$ is bipartite. We show next that $G_{SR}\times H_{SR}$ has a perfect matching. Let  $n_i$ be the order of $G_i$, $i\in \{1,\dots ,k\}$, and let $\{x_1y_1,x_2y_2, ...,x_{|\partial(H)|/2} y_{|\partial(H)|/2}\}\subset E(H_{SR})$ be a perfect matching of  $H_{SR}$.
We distinguish two cases.

\noindent{Case} 1: $G_i$ has a perfect matching. If  $\{u_1v_1,u_2v_2 ...,u_{n_i/2}v_{n_i/2}\}\subset E(G_i)$ is a perfect matching of $G_i$, then the set of edges
\[
\begin{array}{l}
\{(u_1,y_1)\,(v_1,x_1),\; (v_1,y_1)\,(u_1,x_1), ..., (u_{n_i/2}, y_1)\,(v_{n_i/2},x_1),\\
(v_{n_i/2},y_1)\,(u_{n_i/2},x_1), (u_1,y_2)\,(v_1,x_2),\; (v_1,y_2)\,(u_1,x_2), ...,\\
(u_{n_i/2},y_2)\,(v_{n_i/2},x_2),\; (v_{n_i/2},y_2)\,(u_{n_i/2},x_2), \ldots, \\
(u_1,y_{|\partial(H)|/2})\,(v_1,x_{|\partial(H)|/2}),\; (v_1,y_{|\partial(H)|/2})\,(u_1,x_{|\partial(H)|/2}), \ldots, \\
(u_{n_i/2},y_{|\partial(H)|/2})\,(v_{n_i/2},x_{|\partial(H)|/2}),\; (v_{n_i/2}, y_{|\partial(H)|/2})\,(u_{n_i/2},x_{|\partial(H)|/2})\}
\end{array}
 \]
is a perfect matching of $G_i\times H_{SR}$.

\noindent{Case} 2: $G_i$ is Hamiltonian. Let $v_1,v_2,\dots ,v_{n_i},v_1$ be  a Hamiltonian cycle of $G_i$. If $n_i$ is even, then $G_i$ has a perfect matching and this case coincides with Case 1. So we suppose that $n_i$ is odd.  In this case, the set of edges
\[
\begin{array}{l}
\{(v_1,x_1)\,(v_2,y_1),\; (v_2,x_1)\,(v_3,y_1),\dots , (v_{n_i-1},x_1)\,(v_{n_i},y_1),\; (v_{n_i},x_1)\,(v_1,y_1),\\
 (v_1,x_2)\,(v_2,y_2),\; (v_2,x_2)\,(v_3,y_2), ...,(v_{n_i-1},x_2)\,(v_{n_i},y_2),\; (v_{n_i},x_2)\,(v_1,y_2), \ldots,\\
 (v_1,x_{|\partial(H)|/2})\,(v_2,y_{|\partial(H)|/2}),\; (v_2,x_{|\partial(H)|/2})\,(v_3,y_{|\partial(H)|/2}), \ldots,\\
(v_{n_i-1},x_{|\partial(H)|/2})\,(v_{n_i},y_{|\partial(H)|/2}),\; (v_{n_i},x_{|\partial(H)|/2})\,(v_1,y_{|\partial(H)|/2})\}
\end{array}
\]
is a perfect matching of $G_i\times H_{SR}$.

According to Cases 1 and 2 the graph $\bigcup_{i=1}^kG_i\times H_{SR}\cong G_{SR}\times H_{SR}$ has a perfect matching.
\end{proof}

Since $2$-antipodal graphs have strong resolving graphs that are bipartite with a perfect matching, the next result follows from the previous theorem  and Observation \ref{observation1}.

\begin{corollary}
Let $G$ be a $2$-antipodal graph. If $H$ is a $2$-antipodal graph or it is connected and $\partial(H)=\sigma(H)$, then  $(G\Box H)_{SR}$ is bipartite and has a perfect matching.
\end{corollary}

The next well known result characterizing whether Cartesian product graphs are direct product graphs give also an interesting consequence for describing some strong resolving graphs.

\begin{lemma}{\em \cite{Miller1968}}
Let $G$ and $H$ be two connected graphs. Then, $G\Box H\cong G\times H$ if and only if $G\cong H\cong C_{2k+1}$ for some positive integer $k$.
\end{lemma}

The characterization above, Theorem \ref{cartesian_directed} and Observation \ref{observation1}, allow us to immediately determine the strong resolving graph of $C_{2k+1}\times C_{2k+1}$.

\begin{remark}For any nonnegative integer $k$,
$$(C_{2k+1}\times C_{2k+1})_{SR}\cong C_{2k+1}\times C_{2k+1}.$$
\end{remark}

We now turn our attention to describing the structure of the strong resolving graphs of some particular cases of direct product graphs, which in contrast to Cartesian product graphs, becomes more challenging and tedious. Moreover, the results are not stated for general direct product graphs, since it is quite frequently not a connected graphs. From now on, we say that a graph $G$ is 2-\emph{MMD free}, or 2MMF for short, if there exists no pair of MMD vertices $u$ and $v$ with $d_G(u,v)=2$. Clearly diameter two graphs are not 2MMF graphs.

We start the first particular case while describing the structure of the strong resolving graph of $G\times K_n$ for any connected graph $G$. From now on we use the following notation. Consider a set of vertices $V$ and two graphs $G$ and $H$ defined over the sets of vertices $U_1\subseteq V$ and $U_2\subseteq V$, respectively. Hence, $G\sqcup H$ is a graph defined over the set of vertices  $V(G\sqcup H)=U_1\cup U_2$ and $E(G\sqcup H)=E(G)\cup E(H)$. Note that $U_1$ and $U_2$ are not necessarily disjoint, as well as $E(G)$ and $E(H)$. For example, consider a set of seven vertices $v_1,v_2,\ldots,v_7$, the cycle $C_6=v_1v_2\ldots v_6v_1$ and the star $S_{1,6}$ with central vertex in $v_7$ and $v_1,v_2,\ldots,v_6$. Thus, the wheel graph $W_{1,6}$ can be obtained as the graph $C_6\sqcup S_{1,6}$. Another interesting example is for instance the strong product graph $G\boxtimes H$ which can be obtained as $(G\Box H) \sqcup (G\times H)$ (notice that in this case the set of vertices of $G\Box H$ and $G\times H$ coincide).

The following result was recently presented in \cite{Kuziak2016a} as follows.

\begin{theorem}{\rm \cite{Kuziak2016a}}
Let $G$ be a connected 2MMF graph of order at least three and let the integer $n\ge 3$. If
$W$ is the subset of $V(G)$ which contains all vertices belonging to a triangle in $G$, $N_{|W|}$ is the empty graph with vertex set $W$ and the graphs $K_n$, $N_n$ are defined over the same set of vertices, then
$$(G\times K_n)_{SR}\cong (G\Box N_n)\sqcup (G_{SR}\circ N_n)\sqcup (N_{|W|}\Box K_n).$$
\end{theorem}

If we consider $G$ isomorphic to a complete graph $K_r$, then the result above leads to that $(K_r\times K_n)_{SR}\cong (K_r\Box N_n)\sqcup (K_r\circ N_n)\sqcup (N_r\Box K_n)\cong K_r\circ K_n$, which is a contradiction with a result obtained in \cite{RodriguezVelazquez2014a}. Therefore, we next correct the result and, by completeness, include its whole proof, although part of it almost exactly matches that in \cite{Kuziak2016a}.

\begin{theorem}\label{direct-complete}
Let $G$ be a connected 2MMF graph of order at least three and let the integer $n\ge 3$. If
$W$ is the subset of $V(G)$ which contains all vertices belonging to a triangle in $G$, $N_{|W|}$ is the empty graph with vertex set $W$ and the graphs $K_n$, $N_n$ are defined over the same set of vertices, then
$$(G\times K_n)_{SR}\cong\left\{\begin{array}{ll}
  (G\Box N_n)\sqcup (N_{|W|}\Box K_n),
             & \mbox{if $G\cong K_r$ and $r\ge 3$,} \\
       & \\
  (G\Box N_n)\sqcup (G_{SR}\circ N_n)\sqcup (N_{|W|}\Box K_n), & \mbox{otherwise.}
                         \end{array} \right. $$
\end{theorem}

\begin{proof}
Let $(g_1,h_1),(g_2,h_2)$ be two different vertices of $G\times K_n$. We first consider a triangle free graph $G$
and analyze the following possible situations.

\noindent Case 1: $g_1\ne g_2$, $h_1=h_2$ and $g_1\sim g_2$. Hence
$d_{G\times K_n}((g_1,h_1),(g_2,h_1))=3$, since $G$ is triangle free. Also, we observe that
$N_{G\times K_n}(g_1,h_1)=N_G(g_1)\times (V(K_n)-\{h_1\})$ and for every vertex $g\in N_G(g_1)$
and every $h\in V(K_n)-\{h_1\}$ it follows, $d_{G\times K_n}((g_2,h_1),(g,h))=2$. Similarly,
$N_{G\times K_n}(g_2,h_1)=N_G(g_2)\times (V(K_n)-\{h_1\})$ and for every vertex $g\in N_G(g_2)$
and every $h\in V(K_n)-\{h_1\}$ it follows, $d_{G\times K_n}((g_1,h_1),(g,h))=2$. Thus,
$(g_1,h_1)$ and $(g_2,h_2)$ are MMD in $G\times K_n$.

As a consequence of Case 1, for any vertex $h\in V(K_n)$ and any two adjacent vertices
$g,g'$ of $G$, it follows that $(g,h)$ and $(g',h)$ are MMD in $G\times K_n$.
Therefore, the strong resolving graph $(G\times K_n)_{SR}$ contains $n$ copies of $G$ as subgraphs,
or equivalently the graph $G\Box N_n$. We continue describing the other part of $(G\times K_n)_{SR}$.

\noindent Case 2: $g_1\ne g_2$, $g_1\not\sim g_2$ and $g_1,g_2$ are MMD in $G$.
Hence, it follows by Remark \ref{dir-distance} that $d_{G\times K_n}((g_1,h_1),(g_2,h_2))=d_G(g_1,g_2)\geq 3$, since
$G$ is 2MMF graph. It is straightforward
to observe that $(g_1,h_1)$ and $(g_2,h_2)$ are MMD in $G\times K_n$.

As a consequence of Case 2, for any vertices $h,h'\in V(K_n)$ and any two mutually maximally
distant vertices $g,g'$ of $G$, it follows that $(g,h)$ and $(g',h')$ are MMD
in $G\times K_n$. Therefore, the strong resolving graph $(G\times K_n)_{SR}$ contains a subgraph
isomorphic to the lexicographic product of $G_{SR}$ and $N_n$. Next we show that $(G\times K_n)_{SR}$
has no more edges than those ones described until now, which leads to
$(G\times K_n)_{SR}\cong (G\Box N_n)\sqcup (G_{SR}\circ N_n)$.

\noindent Case 3: $g_1\ne g_2$, $g_1\not\sim g_2$ and $g_1,g_2$ are not MMD
in $G$. Similarly to Case 2, it clearly follows that $(g_1,h_1)$ and $(g_2,h_2)$ are not mutually
maximally distant in $G\times K_n$, since for a neighbor $g_3$ of $g_2$ with $d_{G}(g_1,g_3)>d_{G}(g_1,g_2)$
we obtain $d_{G\times K_n}((g_1,h_1),(g_2,h_2))<d_{G\times K_n}((g_1,h_1),(g_3,h))$ for any $h\ne h_2$.

\noindent Case 4: $g_1\ne g_2$, $h_1\ne h_2$ and $g_1\sim g_2$. Hence
$d_{G\times K_n}((g_1,h_1),(g_2,h_2))=1$. Since $n\ge 3$, for any vertex $h_3\notin \{h_1,h_2\}$ we have
that $(g_1,h_3)\in N_{G\times K_n}(g_2,h_2)$ and $d_{G\times K_n}((g_1,h_1),(g_1,h_3))=2$. Thus,
$(g_1,h_1)$ and $(g_2,h_2)$ are not MMD in $G\times K_n$.

\noindent Case 5: $g_1=g_2$. Hence, $d_{G\times K_n}((g_1,h_1),(g_1,h_2))=2$. Since $G$ has order
greater than one, there exists a vertex $g_3\in N_G(g_1)$ and we observe that the vertex
$(g_3,h_1)\in N_{G\times K_n}(g_1,h_2)$. Also, as $G$ is triangle free,
$d_{G\times K_n}((g_1,h_1),(g_3,h_1))=3$. Thus, $(g_1,h_1)$ and $(g_2,h_2)$ are not mutually
maximally distant in $G\times K_n$.

So, if $G$ is triangle free, then we have that $(G\times K_n)_{SR}\cong (G\Box N_n)\sqcup (G_{SR}\circ N_n)$. We consider now that $W$ is the set of vertices of $G$ belonging to a triangle and $|W|=t$. We notice that the fact that there exist vertices belonging to a triangle in $G$ only affects Case 5 and Case 2 (this case is impossible when $G\cong K_r$, $r\ge 3$), and actually it also has an effect on Case 1, but there are no changes in conclusions. That is, if $g_1=g_2$ and $g_1\in W$, then as above $d_{G\times K_n}((g_1,h_1),(g_1,h_2))=2$. However, we have that $N_{G\times K_n}(g_1,h_1)=N_G(g_1)\times (V(K_n)-\{h_1\})$ and for every vertex
$g\in N_G(g_1)$ and every $h\in V(K_n)-\{h_1\}$ it follows, $d_{G\times K_n}((g_2,h_2),(g,h))\le 2$. Similarly, $N_{G\times K_n}(g_2,h_2)=N_G(g_2)\times (V(K_n)-\{h_2\})$ and for every vertex $g\in N_G(g_2)$ and every $h\in V(K_n)-\{h_2\}$ it follows, $d_{G\times K_n}((g_1,h_1),(g,h))\le 2$. Thus, $(g_1,h_1)$ and $(g_1,h_2)$ are MMD in $G\times K_n$.

As a consequence, given a vertex $g\in W$, for any two vertices $h,h'\in V(K_n)$ it follows that $(g,h)$ and $(g,h')$ are MMD in $G\times K_n$. Therefore, the strong resolving graph $(G\times K_n)_{SR}$ contains some other edges than that ones already described for the case of triangle free graphs. These are from $t=|W|$ subgraphs isomorphic to $K_n$, each one corresponding to a vertex in $W$, which is equivalent to the Cartesian product of $N_{|W|}$ (having vertex set $W$) and $K_n$. Moreover, if $G\cong K_r$, $r\ge 3$, then $(G\times K_n)_{SR}\cong (G\Box N_n)\sqcup (N_{|W|}\Box K_n)$, since the situation like in Case 2 is impossible.
\end{proof}

In Figure \ref{Figure-graphs-direct} we exemplify the theorem above. There we give a direct product graphs and its strong resolving graph, drawn in such way we can see all the three subgraphs appearing in the union given in Theorem \ref{direct-complete}, for the case $G$ is not a complete graph.

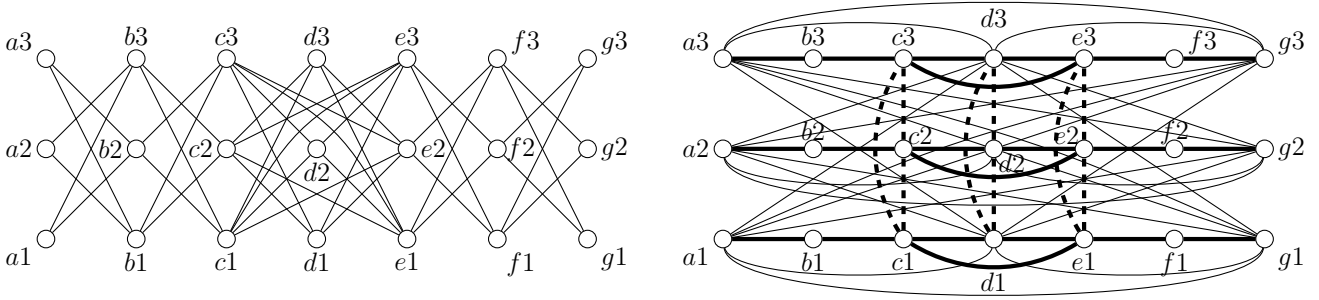
\begin{figure}[ht]
  \centering
\begin{tikzpicture}[scale=.6, transform shape]
\node [draw, shape=circle] (aa1) at  (0,0) {};
\node [draw, shape=circle] (bb1) at  (2,0) {};
\node [draw, shape=circle] (cc1) at  (4,0) {};
\node [draw, shape=circle] (dd1) at  (6,0) {};
\node [draw, shape=circle] (ee1) at  (8,0) {};
\node [draw, shape=circle] (ff1) at  (10,0) {};
\node [draw, shape=circle] (gg1) at  (12,0) {};

\node [draw, shape=circle] (aa2) at  (0,2) {};
\node [draw, shape=circle] (bb2) at  (2,2) {};
\node [draw, shape=circle] (cc2) at  (4,2) {};
\node [draw, shape=circle] (dd2) at  (6,2) {};
\node [draw, shape=circle] (ee2) at  (8,2) {};
\node [draw, shape=circle] (ff2) at  (10,2) {};
\node [draw, shape=circle] (gg2) at  (12,2) {};

\node [draw, shape=circle] (aa3) at  (0,4) {};
\node [draw, shape=circle] (bb3) at  (2,4) {};
\node [draw, shape=circle] (cc3) at  (4,4) {};
\node [draw, shape=circle] (dd3) at  (6,4) {};
\node [draw, shape=circle] (ee3) at  (8,4) {};
\node [draw, shape=circle] (ff3) at  (10,4) {};
\node [draw, shape=circle] (gg3) at  (12,4) {};

\draw(aa1)--(bb2);
\draw(aa1)--(bb3);
\draw(aa2)--(bb1);
\draw(aa2)--(bb3);
\draw(aa3)--(bb1);
\draw(aa3)--(bb2);

\draw(bb1)--(cc2);
\draw(bb1)--(cc3);
\draw(bb2)--(cc1);
\draw(bb2)--(cc3);
\draw(bb3)--(cc1);
\draw(bb3)--(cc2);

\draw(cc1)--(dd2);
\draw(cc1)--(dd3);
\draw(cc2)--(dd1);
\draw(cc2)--(dd3);
\draw(cc3)--(dd1);
\draw(cc3)--(dd2);

\draw(dd1)--(ee2);
\draw(dd1)--(ee3);
\draw(dd2)--(ee1);
\draw(dd2)--(ee3);
\draw(dd3)--(ee1);
\draw(dd3)--(ee2);

\draw(ee1)--(ff2);
\draw(ee1)--(ff3);
\draw(ee2)--(ff1);
\draw(ee2)--(ff3);
\draw(ee3)--(ff1);
\draw(ee3)--(ff2);

\draw(ff1)--(gg2);
\draw(ff1)--(gg3);
\draw(ff2)--(gg1);
\draw(ff2)--(gg3);
\draw(ff3)--(gg1);
\draw(ff3)--(gg2);

\draw(cc1)--(ee2);
\draw(cc2)--(ee1);
\draw(cc2)--(ee3);
\draw(cc3)--(ee2);
\draw(cc1) .. controls (5.5,2.5) .. (ee3);
\draw(cc3) .. controls (6.5,2.5) .. (ee1);

\node [below left, scale=1.4] at (-0.1,0) {$a1$};
\node [left, scale=1.4] at (-0.1,2) {$a2$};
\node [above left, scale=1.4] at (-0.1,4) {$a3$};

\node [below, scale=1.4] at (2,-0.1) {$b1$};
\node [left, scale=1.4] at (1.9,2) {$b2$};
\node [above, scale=1.4] at (2,4.1) {$b3$};

\node [below, scale=1.4] at (4,-0.1) {$c1$};
\node [left, scale=1.4] at (3.9,2) {$c2$};
\node [above, scale=1.4] at (4,4.1) {$c3$};

\node [below, scale=1.4] at (6,-0.1) {$d1$};
\node [below, scale=1.4] at (6,1.9) {$d2$};
\node [above, scale=1.4] at (6,4.1) {$d3$};

\node [below, scale=1.4] at (8,-0.1) {$e1$};
\node [right, scale=1.4] at (8.1,2) {$e2$};
\node [above, scale=1.4] at (8,4.1) {$e3$};

\node [below right, scale=1.4] at (10,-0.1) {$f1$};
\node [right, scale=1.4] at (10,2) {$f2$};
\node [above right, scale=1.4] at (10.1,3.9) {$f3$};

\node [below right, scale=1.4] at (12.1,0) {$g1$};
\node [right, scale=1.4] at (12.1,2) {$g2$};
\node [above right, scale=1.4] at (12.1,3.9) {$g3$};

\node [draw, shape=circle] (a1) at  (15,0) {};
\node [draw, shape=circle] (b1) at  (17,0) {};
\node [draw, shape=circle] (c1) at  (19,0) {};
\node [draw, shape=circle] (d1) at  (21,0) {};
\node [draw, shape=circle] (e1) at  (23,0) {};
\node [draw, shape=circle] (f1) at  (25,0) {};
\node [draw, shape=circle] (g1) at  (27,0) {};

\node [draw, shape=circle] (a2) at  (15,2) {};
\node [draw, shape=circle] (b2) at  (17,2) {};
\node [draw, shape=circle] (c2) at  (19,2) {};
\node [draw, shape=circle] (d2) at  (21,2) {};
\node [draw, shape=circle] (e2) at  (23,2) {};
\node [draw, shape=circle] (f2) at  (25,2) {};
\node [draw, shape=circle] (g2) at  (27,2) {};

\node [draw, shape=circle] (a3) at  (15,4) {};
\node [draw, shape=circle] (b3) at  (17,4) {};
\node [draw, shape=circle] (c3) at  (19,4) {};
\node [draw, shape=circle] (d3) at  (21,4) {};
\node [draw, shape=circle] (e3) at  (23,4) {};
\node [draw, shape=circle] (f3) at  (25,4) {};
\node [draw, shape=circle] (g3) at  (27,4) {};

\node [below left, scale=1.4] at (14.9,0) {$a1$};
\node [left, scale=1.4] at (14.9,2) {$a2$};
\node [above left, scale=1.4] at (14.9,4) {$a3$};
\node [below, scale=1.4] at (17,-0.1) {$b1$};
\node [above, scale=1.4] at (17,2) {$b2$};
\node [above, scale=1.4] at (17,4.1) {$b3$};
\node [below, scale=1.4] at (19,-0.1) {$c1$};
\node [above right, scale=1.4] at (18.9,1.9) {$c2$};
\node [above, scale=1.4] at (19,4.1) {$c3$};
\node [below right, scale=1.4] at (20.5,-0.55) {$d1$};
\node [below right, scale=1.4] at (20.9,2.1) {$d2$};
\node [above, scale=1.4] at (21,4.5) {$d3$};
\node [below, scale=1.4] at (23,-0.1) {$e1$};
\node [above left, scale=1.4] at (23.1,1.9) {$e2$};
\node [above, scale=1.4] at (23,4.1) {$e3$};
\node [below, scale=1.4] at (25,-0.1) {$f1$};
\node [above, scale=1.4] at (25,1.9) {$f2$};
\node [above right, scale=1.4] at (25.1,3.9) {$f3$};
\node [below right, scale=1.4] at (27.1,0) {$g1$};
\node [right, scale=1.4] at (27.1,2) {$g2$};
\node [above right, scale=1.4] at (27.1,3.9) {$g3$};

\draw[ultra thick](a1)--(b1)--(c1)--(d1)--(e1)--(f1)--(g1);
\draw[ultra thick](a2)--(b2)--(c2)--(d2)--(e2)--(f2)--(g2);
\draw[ultra thick](a3)--(b3)--(c3)--(d3)--(e3)--(f3)--(g3);

\draw[ultra thick](c1) .. controls (20.3,-0.8) and (21.7,-0.8) .. (e1);
\draw[ultra thick](c2) .. controls (20.3,1.2) and (21.7,1.2) .. (e2);
\draw[ultra thick](c3) .. controls (20.3,3.2) and (21.7,3.2) .. (e3);

\draw[dashed, ultra thick](c1)--(c2)--(c3);
\draw[dashed, ultra thick](d1)--(d2)--(d3);
\draw[dashed, ultra thick](e1)--(e2)--(e3);

\draw[dashed, ultra thick](c1) .. controls (18.2,1.3) and (18.2,2.7) .. (c3);
\draw[dashed, ultra thick](d1) .. controls (20.2,1.3) and (20.2,2.7) .. (d3);
\draw[dashed, ultra thick](e1) .. controls (22.2,1.3) and (22.2,2.7) .. (e3);

\draw(a1)--(d2);
\draw(a1)--(d3);
\draw(a2)--(d1);
\draw(a2)--(d3);
\draw(a3)--(d1);
\draw(a3)--(d2);

\draw(d1)--(g2);
\draw(d1)--(g3);
\draw(d2)--(g1);
\draw(d2)--(g3);
\draw(d3)--(g1);
\draw(d3)--(g2);

\draw(a1)--(g2);
\draw(a2)--(g1);
\draw(a2)--(g3);
\draw(a3)--(g2);

\draw (a1) .. controls (20,2.2) .. (g3);
\draw (a3) .. controls (22,2.2) .. (g1);

\draw(a1) .. controls (15.2,-1) and (20.8,-1) .. (d1);
\draw(a1) .. controls (15.2,-1.6) and (26.8,-1.6) .. (g1);
\draw(d1) .. controls (21.2,-1) and (26.8,-1) .. (g1);

\draw(a2) .. controls (15.2,1) and (20.8,1) .. (d2);
\draw(a2) .. controls (15.2,0.4) and (26.8,0.4) .. (g2);
\draw(d2) .. controls (21.2,1) and (26.8,1) .. (g2);

\draw(a3) .. controls (15.2,5) and (20.8,5) .. (d3);
\draw(a3) .. controls (15.2,5.6) and (26.8,5.6) .. (g3);
\draw(d3) .. controls (21.2,5) and (26.8,5) .. (g3);
\end{tikzpicture}
\caption{The direct product $H_7\times K_3$ and its strong resolving graph, where $H_7$ is obtained from a path $P_7=abcdefg$ by adding the edge $ce$. According to Theorem \ref{direct-complete}, notice that $W=\{c,d,e\}$. In the strong resolving graph $(H_7\times K_3)_{SR}$: the edges in bold correspond to the subgraph $H_7\Box N_3$ ($N_3$ has vertex set $\{1,2,3\}$); the dashed edges to the subgraph $N_{3}\Box K_3$ ($N_3$ has vertex set $W=\{c,d,e\}$); and the remaining edges to the subgraph $((H_7)_{SR}\circ N_3)\cong (K_3\circ N_3)$ ($K_3$ has vertex set $\{a,d,g\}$).} \label{Figure-graphs-direct}
\end{figure}

For the particular case when $G$ is isomorphic to a complete graph, Theorem \ref{direct-complete} leads to the next corollary.

\begin{corollary}{\rm \cite{RodriguezVelazquez2014a}}\label{k_r-times-k_t}
For any positive integers $r,t\ge 3$,
$$(K_r\times K_t)_{SR}\cong K_r\Box K_t.$$
\end{corollary}

From Theorem \ref{cartesian_directed} and Corollary \ref{k_r-times-k_t} we obtain the following.

\begin{corollary}
For any positive integers $r,t\ge 3$,
$$\left((K_r\times K_t)_{SR}\right)_{SR}\cong K_r\times K_t.$$
\end{corollary}

The following result was also implicitly deduced in \cite[Proof of Theorem 37]{RodriguezVelazquez2014a}, although we now present part of it by using the ideas of Theorem \ref{direct-complete}. To this end, given an odd cycle $C_n = v_0v_1\ldots v_n v_0$, by $C_n^{*}$ we mean the cycle $v_0v_{\left\lfloor n/2\right\rfloor}v_{2\cdot\left\lfloor n/2\right\rfloor}v_{3\cdot\left\lfloor n/2\right\rfloor}\ldots v_{(n-1)\cdot\left\lfloor n/2\right\rfloor}v_0$ where the multiplication operation $x\cdot\left\lfloor n/2\right\rfloor$ with $x\in \{1,\ldots,n-1\}$ is done modulo $n$.

\begin{proposition}\label{c_r-times-k_t}
Let $r\ge 4$ and $t\ge 3$ be positive integers. Let $V(K_t)=V(N_t)=\{v_1,v_2,\dots ,v_{t}\}$ and $C_r=u_0 u_1 \ldots u_{r-1}u_0$ Then the following assertions hold.
\begin{enumerate}[{\rm (i)}]
\item If $r\in \{4,5\}$, then $(C_r\times K_t)_{SR}\cong \bigcup_{i=1}^t K_r$.

\item If $r\ge 6$ is even and $K_2^{(i)}$ is a complete graph on the two vertices $u_i,u_{i+r/2}$ with $i\in \{0,\ldots, r/2-1\}$, then
$$(C_r\times K_t)_{SR}\cong (C_r\Box N_t)\sqcup \left(\bigsqcup_{i=0}^{r/2-1} (K_2^{(i)}\circ N_t)\right).$$

\item If $r\ge 7$, then
$$(C_r\times K_t)_{SR}\cong (C_r\Box N_t)\sqcup (C_r^{*}\circ N_t).$$
\end{enumerate}
\end{proposition}

\begin{proof}
Let $V(K_t)=\{v_1,v_2,\dots ,v_{t}\}$ and $V(C_r)=\{u_0,u_1,\dots ,u_{r-1}\}$, where $u_i\sim u_{i+1}$ for every $i\in \{0,\dots , r-1\}$ and $u_{r-1}\sim u_0$. From now on all the operations with the subscript of a vertex $u_i$ of $C_r$ are expressed modulo $r$. Let $(u_i,v_j),(u_l,v_k)$ be two distinct vertices of $C_r\times K_t$.

(i) Let $r=4$ or $5$. We differentiate four cases.

\noindent{Case} 1: $u_i=u_l$. Hence, $d_{C_r\times K_t}((u_i,v_j),(u_l,v_k))=2$. Since $(u_i,v_j)\sim (u_{i-1},v_k)$, if $k \ne j$ and $d_{C_r\times K_t}((u_{i-1},v_k),(u_l,v_k))=3$, then it follows that  $(u_i,v_j)$ and $(u_l,v_k)$ are not MMD in $C_r\times K_t$.

\noindent{Case} 2: $v_j=v_k$. If $l=i+1$ or $i=l+1$, then without loss of generality we suppose $l=i+1$ and we have that $d_{C_r\times K_t}((u_i,v_j),(u_l,v_k))=3=D(C_r\times K_t)$. Thus, $(u_i,v_j)$ and $(u_l,v_k)$ are MMD in $C_r\times K_t$. On the other hand, if $l\ne i+1$ and $i\ne l+1$, then $d_{C_r\times K_t}((u_i,v_j),(u_l,v_k))=2$. Since for every vertex $(u,v)\in N_{C_r\times K_t}(u_i,v_j)$ we have that $d_{C_r\times K_t}((u,v),(u_l,v_k))\le 2$ and also for every vertex $(u,v)\in N_{C_r\times K_t}(u_l,v_k)$ we have that $d_{C_r\times K_t}((u,v),(u_i,v_j))\le 2$, we obtain that $(u_i,v_j)$ and $(u_l,v_k)$ are MMD in $C_r\times K_t$.

\noindent{Case} 3: $u_i\ne u_l$, $v_j\ne v_k$ and $(u_i,v_j)\sim (u_l,v_k)$. So, there exists a vertex $(u,v)\in N_{C_r\times K_t}(u_l,v_k)$ such that $d_{C_r\times K_t}((u,v),(u_i,v_j))=2$ and, as a consequence, $(u_i,v_j)$ and $(u_l,v_k)$ are not MMD in $C_r\times K_t$.

\noindent{Case} 4: $u_i\ne u_l$, $v_j\ne v_k$ and $(u_i,v_j)\not\sim (u_l,v_k)$. Hence, we have that $d_{C_r\times K_t}((u_i,v_j),(u_l,v_k))=2$. We can suppose, without loss of generality, that $l=i+2$. Since

\begin{itemize}
\item $(u_i,v_j)\sim (u_{l-1},v_k)$ and $(u_l,v_k)\sim (u_{l-1},v_j)$ and also,
\item $d_{C_r\times K_t}((u_i,v_j),(u_{l-1},v_j))=3$ and $d_{C_r\times K_t}((u_l,v_k),(u_{l-1},v_k))=3$,
\end{itemize}
we obtain that $(u_i,v_j)$ and $(u_l,v_k)$ are not MMD in $C_r\times K_t$. Hence the strong resolving graph $(C_r\times K_t)_{SR}$ is isomorphic to $\bigcup_{i=1}^t K_r$.

(ii) Let $r\ge 6$. The result is a direct consequence of Theorem \ref{direct-complete}.
\end{proof}

The first item of next result was implicitly deduced in \cite[Proof of Theorem 38]{RodriguezVelazquez2014a}.


\begin{proposition}\label{p_r-times-k_t}
Let $r\ge 2$ and $t\ge 3$ be positive integers. Then the following assertions hold.
\begin{enumerate}[{\rm (i)}]
\item If   $r\in \{2,3\}$, then $(P_r\times K_t)_{SR}\cong \bigcup_{i=1}^t K_r$.

\item If $r\ge 4$, $P_r=v_1v_2\ldots v_r$ and $P_2=v_1v_r$, then $(P_r\times K_t)_{SR}\cong (P_r\Box N_t)\sqcup (P_2\circ N_t)$.
\end{enumerate}
\end{proposition}

\begin{proof}
Let $V(K_t)=\{v_1,v_2,\dots ,v_{t}\}$ and $V(P_r)=\{u_0,u_1,\dots ,u_{r-1}\}$, where $u_i\sim u_{i+1}$ for every $i\in \{0,\dots , r-1\}$.

If $r=2$, then a vertex $(u_i,v_j)$ in $P_2\times K_t$ is MMD only with the vertex $(u_l, v_j)$, where $i\ne l$. So, $(P_2\times K_t)_{SR}\cong \bigcup_{m=1}^{t} K_2$.

If $r=3$, then a vertex $(u_i,v_j)$ in $P_3\times K_t$ is MMD only with those vertices $(u_l, v_j)$, where $i\ne l$. Thus, $(P_3\times K_t)_{SR}\cong \bigcup_{m=1}^{t} K_3$.

If $r\ge 4$, then the result is a particular case of Theorem \ref{direct-complete}.
\end{proof}

We next deal with the direct product of a complete bipartite graph and a complete graph. In contrast with Theorem \ref{direct-complete}, in this case all the MMD vertices of the complete bipartite graph are at distance two.

\begin{remark}{\rm \cite{Kuziak2016a}}\label{lem-K_n-K_r-t}
For any $r,t\ge 1$ and any $n\ge 3$, $$(K_{r,t}\times K_n)_{SR}\cong \bigcup_{i=1}^n K_{r+t}.$$
\end{remark}

\begin{proof}
Let $X,Y$ be the bipartition sets of $K_{r,t}$ such that $|X|=r$ and $|Y|=t$. Consider the vertices $g\in X$ and $h\in V(K_n)$. We notice that vertices in $A=Y\times (V(K_n)-\{h\})$ form the open neighborhood of $(g,h)$. Since $n\ge 3$, every vertex from $(X\times V(K_n))-\{(g,h)\})$ has a neighbor in $A$ and viceversa. Thus, vertices of $A$ are not MMD with $(g,h)$. On the other hand, the remaining vertices are $Y\times \{h\}$ and they are adjacent to all vertices in $X\times (V(K_n)-\{h\})$. Clearly, any vertex in $Y\times \{h\}$ is MMD with $(g,h)$. Moreover, the vertices in $X\times (V(K_n)-\{h\})$ are not MMD with $(g,h)$. Finally, we notice that the vertices in $(X-\{g\})\times \{h\}$ are not adjacent to any vertex in $Y\times \{h\}$. So, every vertex in $(X-\{g\})\times \{h\}$ is MMD with $(g,h)$. As a consequence, $(g,h)$ is adjacent in $(K_{r,t}\times K_n)_{SR}$ to every vertex of $(V(K_{r,t})-\{g\})\times \{h\}$. Therefore, by symmetry, the proof is completed.
\end{proof}

Now we present some results for graphs of diameter two as factors of a direct product. Since it is necessary to be careful with connectedness of the direct product, the results are separated with respect to whether one factor is bipartite or not. It is not hard to see that the only bipartite graphs of diameter two are the complete bipartite graphs $K_{k,\ell}$, where $\max \{k,\ell\}\geq 2$.

Another important measure for the strong resolving graphs of a direct product of two graphs of diameter two is when the factors are triangle free and moreover, when every pair of vertices is on a five-cycle. Hence, we call a graph in which every pair of vertices is on a common five-cycle, a $C_5$-\emph{connected} graph. Clearly, a $C_5$-connected graph has diameter at most two. Moreover, if $G$ is a triangle free $C_5$-connected graph, then its diameter equals two. The Petersen graph is $C_5$-connected triangle free graph. The graph $G$ of Figure \ref{fig:C6} is an example of a triangle free graph of diameter two in which $u$ and $v$ are not on a common five-cycle and $G$ is not $C_5$-connected. The graph $H$ of the same figure is a
triangle free $C_5$-connected graph of diameter two.

\begin{figure}[ht!]
\begin{center}
\begin{tikzpicture}[scale=0.6,style=thick]
\def\vr{6pt}
\path (5,-3) coordinate (a); \path (8,0) coordinate (b);
\path (5,3) coordinate (c); \path (2,0) coordinate (d);
\path (4,0) coordinate (e); \path (5,1) coordinate (f);
\path (6,0) coordinate (g); \path (5,-1) coordinate (h);

\path (-3,-2) coordinate (v); \path (-4,0.5) coordinate (w);
\path (-1.5,0.5) coordinate (x); \path (-4,3) coordinate (y);
\path (-5,-2) coordinate (u); \path (-6.5,0.5) coordinate (z);

\draw (a) -- (b); \draw (g) -- (h);
\draw (c) -- (b); \draw (c) -- (f);
\draw (b) -- (g); \draw (a) -- (h);
\draw (a) -- (d); \draw (c) -- (d);
\draw (d) -- (e); \draw (f) -- (e);
\draw (g) -- (e); \draw (f) -- (h);

\draw (z) -- (u);
\draw (u) -- (v); \draw (u) -- (w);
\draw (v) -- (x); \draw (x) -- (y);
\draw (y) -- (z); \draw (y) -- (w);
\draw (a)  [fill=white] circle (\vr); \draw (b)  [fill=white] circle (\vr);
\draw (c)  [fill=white] circle (\vr); \draw (d)  [fill=white] circle (\vr);
\draw (e)  [fill=white] circle (\vr); \draw (f)  [fill=white] circle (\vr);
\draw (g)  [fill=white] circle (\vr); \draw (h)  [fill=white] circle (\vr);
\draw (u)  [fill=white] circle (\vr); \draw (v)  [fill=white] circle (\vr);
\draw (w)  [fill=white] circle (\vr); \draw (x)  [fill=white] circle (\vr);
\draw (y)  [fill=white] circle (\vr); \draw (z)  [fill=white] circle (\vr);
\draw (2,3) node {$H$}; \draw (-6,3) node {$G$};
\draw (-7,0.5) node {$u$}; \draw (-3.5,0.5) node {$v$};
\end{tikzpicture}
\end{center}
\caption{Two triangle free graphs of diameter two.} \label{fig:C6}
\end{figure}
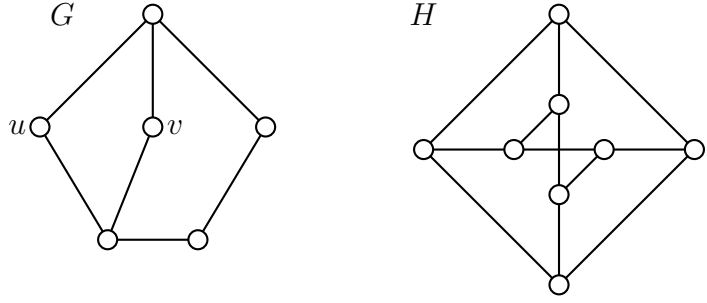

\begin{theorem}{\rm \cite{Kuziak2016a}}\label{complbip}
Let $G$ be a nonbipartite triangle free graph of order $n\ge 2$ and let $\max\{k,\ell\}\geq 2$. If $G$ is $C_5$-connected, then
$$(G\times K_{k,\ell})_{SR}\cong N_n\Box K_{k+\ell}.$$
\end{theorem}

\begin{proof}
Let $V(G)=\{g_1,\ldots ,g_n\}$ and $U(K_{k,\ell})=U_1\cup U_2$ where $U_1=\{u_1,\ldots ,u_k\}$ and $U_2=\{v_1,\ldots ,v_{\ell}\}$.
Clearly, $d^e_{K_{k,\ell}}(u_i,v_j)=\infty$, $d^o_{K_{k,\ell}}(u_i,v_j)=1$, $d^o_{K_{k,\ell}}(v_i,v_j)=\infty$ and
$d^o_{K_{k,\ell}}(u_i,u_j)=\infty$ for any $i$ and $j$. Also, $d^e_{K_{k,\ell}}(u_i,u_j)=2$ and $d^e_{K_{k,\ell}}(v_i,v_j)=2$
for every $i\neq j$. Conversely, by $C_5$-connectedness of $G$, $d^e_G(g_i,g_j)$ and $d^o_G(g_i,g_j)$ always exist.
Moreover, $d^e_G(g_i,g_j)$ is between 0 and 4, while $d^o_G(g_i,g_j)$ is between 1 and 5. Hence, by the distance formula presented in
Remark \ref{dir-distance} we can have the distances between 0 and 5 in $G\times K_{k,\ell}$.
Again, by this distance formula, it is easy to see that $d_{G\times K_{k,\ell}}((g_1,u_1),(g_1,v_j))=5$ for any $j\in \{1,\ldots ,\ell\}$
and that $d_{G\times K_{k,\ell}}((g_1,u_1),(g_1,u_j))=2$ for any $j\in \{2,\ldots ,k\}$. We show that vertices satisfying these equalities above are the only neighbors of $(g_1,u_1)$ in the strong resolving graph $(G\times K_{k,\ell})_{SR}$. Clearly, $(g_1,u_1)$ and $(g_1,v_j)$ are MMD, since they are
diametral vertices for any $j\in \{1,\ldots ,\ell\}$. Since $N_{K_{k,\ell}}(u_1)=N_{K_{k,\ell}}(u_j)$, for any $j\in \{2,\ldots ,k\}$,
by equation \eqref{neigh} that describes  neighborhoods in the direct product, we see that $(g_1,u_1)$ and $(g_1,u_j)$ have the same neighborhood and therefore, they are MMD.

Next we show that no other vertex of $G\times K_{k,\ell}$ is MMD with $(g_1,u_1)$. In this case, we reduce it
to a five-cycle, since $G$ is $C_5$-connected. We may assume that $g_1g_2g_3g_4g_5g_1$ is a five-cycle. By the symmetry of a five-cycle
we need to present the arguments only for $g_2$ and $g_3$. For every $j\in \{1,\ldots ,\ell\}$ and $i\in \{2,\ldots ,\ell\}$ they are as follows:
\begin{itemize}
\item $(g_2,v_j)\sim (g_3,u_1)$ and $(g_2,v_j)$ is closer to $(g_1,u_1)$ than $(g_3,u_1)$;
\item $(g_2,u_i)\sim (g_1,v_1)$ and $(g_2,u_i)$ is closer to $(g_1,u_1)$ than $(g_1,v_1)$;
\item $(g_3,v_j)\sim (g_2,u_1)$ and $(g_3,v_j)$ is closer to $(g_1,u_1)$ than $(g_2,u_1)$;
\item $(g_3,u_i)\sim (g_4,v_1)$ and $(g_3,u_i)$ is closer to $(g_1,u_1)$ than $(g_4,v_1)$.
\end{itemize}
See the graph $C_5\times K_{1,2}\cong C_5\times P_3$ on the left part of Figure \ref{second}, where the distances from $(g_1,u_1)$ are marked.
Thus, the vertex $(g_1,u_1)$ is adjacent to all vertices of $\{g_1\}\times (V(K_{k,\ell})-\{u_1\})$ in $(G\times K_{k,\ell})_{SR}$.
Notice that the same argument also holds when $\min\{k,\ell\}=1$.
We can use the same arguments for any vertex of $G\times K_{k,\ell}$ and therefore, we have $(G\times K_{k,\ell})_{SR}\cong N_n\Box K_{k+\ell}$.
\end{proof}


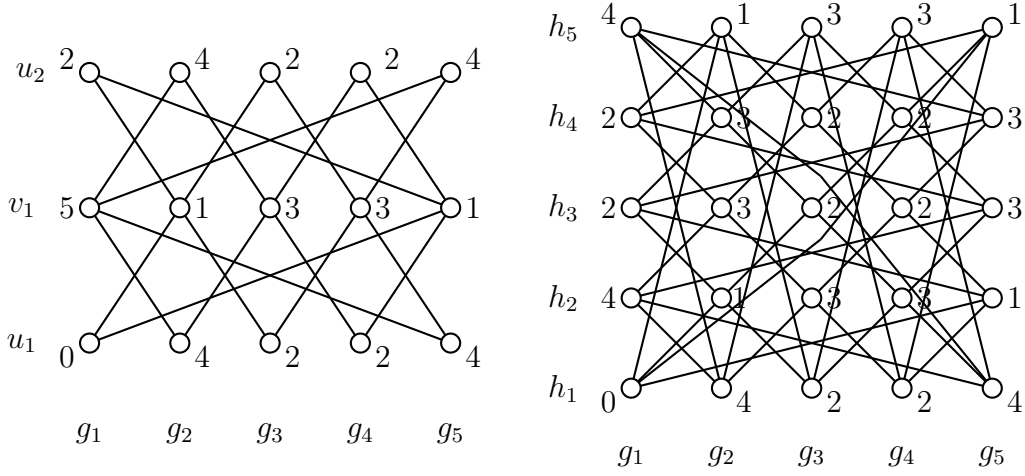
\begin{figure}[ht!]
\begin{center}
\begin{tikzpicture}[scale=0.6,style=thick]
\def\vr{6pt}
\path (2,-4) coordinate (a1); \path (4,-4) coordinate (a2);
\path (6,-4) coordinate (a3); \path (8,-4) coordinate (a4);
\path (10,-4) coordinate (a5); \path (2,-2) coordinate (b1);
\path (4,-2) coordinate (b2); \path (6,-2) coordinate (b3);
\path (8,-2) coordinate (b4); \path (10,-2) coordinate (b5);
\path (2,0) coordinate (c1); \path (4,0) coordinate (c2);
\path (6,0) coordinate (c3); \path (8,0) coordinate (c4);
\path (10,0) coordinate (c5); \path (2,2) coordinate (d1);
\path (4,2) coordinate (d2); \path (6,2) coordinate (d3);
\path (8,2) coordinate (d4); \path (10,2) coordinate (d5);
\path (2,4) coordinate (e1); \path (4,4) coordinate (e2);
\path (6,4) coordinate (e3); \path (8,4) coordinate (e4);
\path (10,4) coordinate (e5);
\path (6.2,-0.7) coordinate (x); \path (6.2,0.7) coordinate (y);

\path (-10,-3) coordinate (u1); \path (-8,-3) coordinate (u2);
\path (-6,-3) coordinate (u3); \path (-4,-3) coordinate (u4);
\path (-2,-3) coordinate (u5); \path (-10,0) coordinate (v1);
\path (-8,0) coordinate (v2); \path (-6,0) coordinate (v3);
\path (-4,0) coordinate (v4); \path (-2,0) coordinate (v5);
\path (-10,3) coordinate (w1); \path (-8,3) coordinate (w2);
\path (-6,3) coordinate (w3); \path (-4,3) coordinate (w4);
\path (-2,3) coordinate (w5);

\draw (a1) -- (b2); \draw (a1) -- (e2);
\draw (a1) -- (b5); 
\draw (a2) -- (b1); \draw (a2) -- (e1);
\draw (a2) -- (b3); \draw (a2) -- (e3);
\draw (a3) -- (b2); \draw (a3) -- (e2);
\draw (a3) -- (b4); \draw (a3) -- (e4);
\draw (a4) -- (b3); \draw (a4) -- (e3);
\draw (a4) -- (b5); \draw (a4) -- (e5);
\draw (a5) -- (b1); 
\draw (a5) -- (b4); \draw (a5) -- (e4);
\draw (b1) -- (c2); \draw (b1) -- (c5);
\draw (b2) -- (c1); \draw (b2) -- (c3);
\draw (b3) -- (c2); \draw (b3) -- (c4);
\draw (b4) -- (c3); \draw (b4) -- (c5);
\draw (b5) -- (c1); \draw (b5) -- (c4);
\draw (c1) -- (d2); \draw (c1) -- (d5);
\draw (c2) -- (d1); \draw (c2) -- (d3);
\draw (c3) -- (d2); \draw (c3) -- (d4);
\draw (c4) -- (d3); \draw (c4) -- (d5);
\draw (c5) -- (d1); \draw (c5) -- (d4);
\draw (d1) -- (e2); \draw (d1) -- (e5);
\draw (d2) -- (e1); \draw (d2) -- (e3);
\draw (d3) -- (e2); \draw (d3) -- (e4);
\draw (d4) -- (e3); \draw (d4) -- (e5);
\draw (d5) -- (e1); \draw (d5) -- (e4);
\draw (e5) -- (x); \draw (x) -- (a1);
\draw (y) -- (e1); \draw (a5) -- (y);

\draw (u1) -- (v2); \draw (u1) -- (v5);
\draw (u2) -- (v1); \draw (u2) -- (v3);
\draw (u3) -- (v2); \draw (u3) -- (v4);
\draw (u4) -- (v3); \draw (u4) -- (v5);
\draw (u5) -- (v4); \draw (u5) -- (v1);
\draw (v1) -- (w2); \draw (v1) -- (w5);
\draw (v2) -- (w3); \draw (v2) -- (w1);
\draw (v3) -- (w2); \draw (v3) -- (w4);
\draw (v4) -- (w3); \draw (v4) -- (w5);
\draw (v5) -- (w1); \draw (v5) -- (w4);
\draw (a1)  [fill=white] circle (\vr); \draw (a2)  [fill=white] circle (\vr);
\draw (a3)  [fill=white] circle (\vr); \draw (a4)  [fill=white] circle (\vr);
\draw (a5)  [fill=white] circle (\vr); \draw (b1)  [fill=white] circle (\vr);
\draw (b2)  [fill=white] circle (\vr); \draw (b3)  [fill=white] circle (\vr);
\draw (b4)  [fill=white] circle (\vr); \draw (b5)  [fill=white] circle (\vr);
\draw (c1)  [fill=white] circle (\vr); \draw (c2)  [fill=white] circle (\vr);
\draw (c3)  [fill=white] circle (\vr); \draw (c4)  [fill=white] circle (\vr);
\draw (c5)  [fill=white] circle (\vr); \draw (d1)  [fill=white] circle (\vr);
\draw (d2)  [fill=white] circle (\vr); \draw (d3)  [fill=white] circle (\vr);
\draw (d4)  [fill=white] circle (\vr); \draw (d5)  [fill=white] circle (\vr);
\draw (e1)  [fill=white] circle (\vr); \draw (e2)  [fill=white] circle (\vr);
\draw (e3)  [fill=white] circle (\vr); \draw (e4)  [fill=white] circle (\vr);
\draw (e5)  [fill=white] circle (\vr);

\draw (u1)  [fill=white] circle (\vr);
\draw (u2)  [fill=white] circle (\vr); \draw (u3)  [fill=white] circle (\vr);
\draw (u4)  [fill=white] circle (\vr); \draw (u5)  [fill=white] circle (\vr);
\draw (v1)  [fill=white] circle (\vr); \draw (v2)  [fill=white] circle (\vr);
\draw (v3)  [fill=white] circle (\vr); \draw (v4)  [fill=white] circle (\vr);
\draw (v5)  [fill=white] circle (\vr); \draw (w1)  [fill=white] circle (\vr);
\draw (w2)  [fill=white] circle (\vr); \draw (w3)  [fill=white] circle (\vr);
\draw (w4)  [fill=white] circle (\vr); \draw (w5)  [fill=white] circle (\vr);

\draw (-10.5,-3.3) node {$0$}; \draw (-7.5,-3.3) node {$4$};
\draw (-5.5,-3.3) node {$2$}; \draw (-3.5,-3.3) node {$2$};
\draw (-1.5,-3.3) node {$4$}; \draw (-10.5,0) node {$5$};
\draw (-7.5,0) node {$1$}; \draw (-5.5,0) node {$3$};
\draw (-3.5,0) node {$3$}; \draw (-1.5,0) node {$1$};
\draw (-10.5,3.3) node {$2$}; \draw (-7.5,3.3) node {$4$};
\draw (-5.5,3.3) node {$2$}; \draw (-3.3,3.3) node {$2$};
\draw (-1.5,3.3) node {$4$}; \draw (-10,-5) node {$g_1$};
\draw (-8,-5) node {$g_2$}; \draw (-6,-5) node {$g_3$};
\draw (-4,-5) node {$g_4$}; \draw (-2,-5) node {$g_5$};
\draw (-11.5,-3) node {$u_1$}; \draw (-11.3,3) node {$u_2$};
\draw (-11.5,0) node {$v_1$}; 

\draw (1.5,-4.3) node {$0$}; \draw (4.5,-4.3) node {$4$};
\draw (6.5,-4.3) node {$2$}; \draw (8.5,-4.3) node {$2$};
\draw (10.5,-4.3) node {$4$}; \draw (1.5,-2) node {$4$};
\draw (4.4,-2) node {$1$}; \draw (6.5,-2) node {$3$};
\draw (8.5,-2) node {$3$}; \draw (10.5,-2) node {$1$};
\draw (1.5,0) node {$2$}; \draw (4.5,0) node {$3$};
\draw (6.5,0) node {$2$}; \draw (8.5,0) node {$2$};
\draw (10.5,0) node {$3$}; \draw (2,-5.5) node {$g_1$};
\draw (4,-5.5) node {$g_2$}; \draw (6,-5.5) node {$g_3$};
\draw (8,-5.5) node {$g_4$}; \draw (10,-5.5) node {$g_5$};
\draw (0.5,-4) node {$h_1$}; \draw (0.5,-2) node {$h_2$};
\draw (0.5,0) node {$h_3$}; \draw (0.5,2) node {$h_4$};
\draw (0.5,4) node {$h_5$}; \draw (1.5,2) node {$2$};
\draw (4.5,2) node {$3$}; \draw (6.5,2) node {$2$};
\draw (8.5,2) node {$2$}; \draw (10.5,2) node {$3$};
\draw (1.5,4.3) node {$4$}; \draw (4.5,4.3) node {$1$};
\draw (6.5,4.3) node {$3$}; \draw (8.5,4.3) node {$3$};
\draw (10.5,4.3) node {$1$};
\end{tikzpicture}
\end{center}
\caption{Situations from the proofs of Theorems \ref{complbip} and \ref{C5connected}.} \label{second}
\end{figure}

\begin{theorem}{\rm \cite{Kuziak2016a}}\label{C5connected}
For any nonbipartite triangle free $C_5$-connected graphs $G$ and $H$ of diameter two, $$(G\times H)_{SR}\cong G\Box H.$$
\end{theorem}

\begin{proof}
Let $V(G)=\{g_1,\ldots ,g_n\}$ and $V(H)=\{h_1,\ldots ,h_k\}$. Note that $G$ and $H$ are $C_5$-connected graphs, which implies that their even and odd distances between arbitrary vertices always exist. Moreover, the even distances are between 0 and 4, while the odd distances are between 1 and 5. Now, according to Remark \ref{dir-distance}, the distances in $G\times H$ are between 0 and 4. We may assume that $g_1g_2g_3g_4g_5g_1$
and $h_1h_2h_3h_4h_5h_1$ are induced five-cycles of triangle free $C_5$-connected graphs $G$ and $H$, respectively. Again, by this
distance formula, it is easy to see that $d_{G\times H}((g_1,h_1),(g_1,h_j))=4$ for $j\in \{2,5\}$ and that $d_{G\times H}((g_1,h_1),(g_j,h_1))=4$ for
$j\in \{2,5\}$. We show that these are the only neighbors of $(g_1,u_1)$ in $(G\times H)_{SR}$. Clearly, these pairs are mutually
maximally distant since they are diametral vertices.

We now show that $(g_1,u_1)$ is not MMD with any other vertex of $G\times K_{k,\ell}$. By the symmetry of a five-cycle
and the commutativity of the direct product we need to present the arguments only for $g_1,g_2$ and $g_3$ and for $h_1,h_2$ and $h_3$.
They are as follows:
\begin{itemize}
\item $(g_1,h_3)\sim (g_2,h_4)$ and $(g_1,h_3)$ is closer to $(g_1,h_1)$ than $(g_2,h_4)$;
\item $(g_2,h_2)\sim (g_3,h_1)$ and $(g_2,h_2)$ is closer to $(g_1,h_1)$ than $(g_3,h_1)$;
\item $(g_2,h_3)\sim (g_1,h_2)$ and $(g_2,h_3)$ is closer to $(g_1,h_1)$ than $(g_1,h_2)$;
\item $(g_3,h_1)\sim (g_4,h_2)$ and $(g_3,h_1)$ is closer to $(g_1,h_1)$ than $(g_4,h_2)$;
\item $(g_3,h_2)\sim (g_2,h_1)$ and $(g_3,h_2)$ is closer to $(g_1,h_1)$ than $(g_2,h_1)$;
\item $(g_3,h_3)\sim (g_2,h_4)$ and $(g_3,h_3)$ is closer to $(g_1,h_1)$ than $(g_2,h_4)$.
\end{itemize}
See the graph $C_5\times C_5$ on the right part of Figure \ref{second}, where the distances from $(g_1,h_1)$ are marked.
So, the vertex $(g_1,h_1)$ is adjacent to $(g_1,h_2),(g_1,h_5),(g_2,h_1)$ and $(g_5,h_1)$ in $(G\times K_{k,\ell})_{SR}$.
Continuing with the same arguments, we obtain that $(g_1,u_1)$ is adjacent to all vertices of
$(\{g_1\}\times N_H(h_1))\cup (N_G(g_1)\times \{h_1\})$ in $(G\times H)_{SR}$.
By using the same arguments for any vertex of $G\times H$ we obtain $(G\times H)_{SR}\cong G\Box H$, which completes the proof.
\end{proof}

\subsection{Cartesian sum and Strong product graphs}

The description of the strong resolving graph of $G\oplus H$ can be easily obtained from Proposition~\ref{proposition diameter two} and
Proposition~\ref{lem Cart sum diam}.

\begin{proposition}{\rm \cite{Kuziak2014b}}\label{SRGraphDianCartSumle2}
Let $G$ and $H$ be two nontrivial graphs such that at least one of them is noncomplete. If $D(G)\le 2$
or neither $G$ nor  $H$ has isolated vertices, then
$$(G\oplus H)_{SR}\cong (G\oplus H)^*_-.$$
\end{proposition}

\begin{proof}
We assume that $D(G)\le 2$
or neither $G$ nor  $H$ has isolated vertices. Then, by Proposition \ref{lem Cart sum diam} we have $D(G\oplus H)=2$ and hence, by Proposition~\ref{proposition diameter two}, $(G\oplus H)_{SR}\cong (G\oplus H)^*_-$.
\end{proof}

We now describe the structure of the strong resolving graph of $G\boxtimes H$.

\begin{lemma}{\rm \cite{Kuziak2013c}}\label{lem boundary}
Let $G$ and $H$ be two connected nontrivial graphs. Let $u,x$ be two vertices of $G$ and let $v,y$ be two vertices of $H$. Then $(u,v)$ and $(x,y)$ are MMD vertices in $G\boxtimes H$ if and only if one of the following conditions holds:
\begin{enumerate}[{\rm (i)}]
\item $u,x$ are MMD in $G$ and $v,y$ are MMD in $H$;
\item $u,x$ are MMD in $G$ and $v=y$;
\item $v,y$ are MMD in $H$ and $u=x$;
\item $u,x$ are MMD in $G$ and $d_G(u,x) > d_H(v,y)$;
\item $v,y$ are MMD in $H$ and $d_G(u,x) < d_H(v,y)$.
\end{enumerate}
\end{lemma}

We need to introduce more notation. Let $G=(V,E)$ and $G'=(V',E')$ be two graphs. If $V'\subseteq V$ and $E'\subseteq E$, then $G'$ is a subgraph of $G$ and we denote that by $G'\sqsubseteq G$\label{g subgraph}. Notice that Lemma \ref{lem boundary} leads to the following relationship.

\begin{theorem}{\rm \cite{Kuziak2013c}}\label{th products subgraphs}
For any connected graphs $G$ and $H$,
$$G_{SR+I}\boxtimes H_{SR+I}\sqsubseteq (G\boxtimes H)_{SR+I}\sqsubseteq G_{SR+I}\oplus H_{SR+I}.$$
\end{theorem}

\begin{proof}
Notice that $$V(G_{SR+I}\boxtimes H_{SR+I})=V((G\boxtimes H)_{SR+I})=V(G_{SR+I}\oplus H_{SR+I})=V(G)\times V(H).$$ Let $(u,v)$ and $(x,y)$ be two vertices adjacent in $G_{SR+I}\boxtimes H_{SR+I}$. So, either
\begin{itemize}
\item $u=x$ and $vy\in E(H_{SR+I})$, or
\item $ux\in E(G_{SR+I})$ and $v=y$, or
\item $ux\in E(G_{SR+I})$ and $vy\in E(H_{SR+I})$.
\end{itemize}
Hence, by using respectively the condition (iii), (ii) and (i) of Lemma \ref{lem boundary} we have that $(u,v)$ and $(x,y)$ are also adjacent in $(G\boxtimes H)_{SR+I}$.

Now, let $(u',v')$ and $(x',y')$ be two vertices adjacent in $(G\boxtimes H)_{SR+I}$. From Lemma \ref{lem boundary} we obtain that $u'x'\in E(G_{SR+I})$ or $v'y'\in E(H_{SR+I})$. Thus, $(u',v')$ and $(x',y')$ are also adjacent in $G_{SR+I}\oplus H_{SR+I}$.
\end{proof}

\subsection{Lexicographic product graphs}\label{SectionLexicographic}


From the next lemmas we can describe the structure of the strong resolving graph of $G\circ H$.

\begin{lemma}{\rm \cite{Kuziak2014}}\label{lem mmd}
Let $G$ be a connected nontrivial graph and let $H$ be a nontrivial graph. Let $a,b\in V(G)$ be such that they are not true twin vertices and let $x,y\in V(H)$. Then $(a,x)$ and $(b,y)$ are MMD in $G\circ H$ if and only if $a$ and $b$ are MMD in $G$.
\end{lemma}

\begin{proof}
Let $x,y\in V(H)$. We assume that $a,b\in V(G)$ are MMD in $G$ and that they are not true twins. 
 First of all, notice that $d_G(a,b)\ge 2$, (if $d_G(a,b)=1$, then to be MMD in $G$, they must be true twins).
 Hence, by Theorem \ref{basictoolLexicographic} (i) we have that if
$(c,d)\in N_{G\circ H}(b,y)$, then either $c=b$ or $c\in N_G(b)$. In both cases, by Theorem \ref{basictoolLexicographic} (ii) we obtain $d_{G\circ H}((a,x),(c,d))=d_G(a,c)\le d_G(a,b)=d_{G\circ H}((a,x),(b,y))$. So, $(b,y)$ is maximally distant from $(a,x)$ and, by symmetry, we conclude that $(b,y)$ and $(a,x)$ are MMD in $G\circ H$.

Conversely, assume that $(a,x)$ and $(b,y)$, $a\ne b$, are MMD in $G\circ H$. If  $c\in N_G(b)$, then for any $z\in V(H)$ we have $(c,z)\in N_{G\circ H}(b,y)$. Now,  by Theorem \ref{basictoolLexicographic} (ii) we obtain $d_G(a,c)=d_{G\circ H}((a,x),(c,z))\le d_{G\circ H}((a,x),(b,y))=d_G(a,b)$.
So, $b$ is maximally distant from $a$ and, by symmetry, we conclude that $b$ and $a$ are MMD in $G$.
\end{proof}

\begin{lemma}{\rm \cite{Kuziak2014}}\label{lemmaTrueTwins}
Let $G$ be a connected nontrivial  graph, let $H$ be a graph of order $n\ge 2$, let $a,b\in V(G)$ be two distinct true twin vertices and let $x,y\in V(H)$. Then $(a,x)$ and $(b,y)$ are  MMD in $G\circ H$ if and only if both, $x$ and $y$, have degree $n-1$.
\end{lemma}

\begin{proof}
If  $x\in V(H)$ has degree $n-1$, then for any  $y\in V(H)$ of degree $n-1$ we have that
$(a,x)$  and $(b,y)$ are true twins in $G\circ H$. Hence, $(a,x)$ and $(b,y)$ are  MMD in $G\circ H$.

Now, suppose that  there exists $z\in V(H)-N_H(x)$. By Theorem \ref{basictoolLexicographic} (iii), it follows that $d_{G\circ H}((a,x),(a,z))=2$. Also, for every $y\in V(H)$, Theorem \ref{basictoolLexicographic} (ii) gives $d_{G\circ H}((a,x),(b,y))=1$. Thus, we conclude that $(a,x)$ and $(b,y)$ are not MMD in $G\circ H$.
\end{proof}

The strong resolving graph of the Lexicographic product can be described using graphs $G^*$ and $G^*_-$ already defined in Section~\ref{SectionDetermination problem}.

\begin{remark}{\rm \cite{Kuziak2014}}\label{rem G*}
 Let $G$ be a connected graph of diameter $D(G)$, order $n$ and maximum degree $\Delta(G).$
\begin{enumerate}[{\rm (i)}]
\item If $\Delta(G)\le n-2$, then $G^*\cong (K_1+G)_{SR}$. 
\item If $D(G)\le 2$, then $G^*_-\cong G_{SR}$.

\item If $G$ has no true twins, then $G^*\cong G^c$.
\end{enumerate}
\end{remark}

\begin{lemma}{\rm \cite{Kuziak2014}}\label{lem mmd in copy H}
Let $G$ be a connected nontrivial graph. Let $x,y\in V(H)$ be two distinct vertices of a graph $H$  and let $a\in V(G)$. Then $(a,x)$ and $(a,y)$ are MMD vertices in $G\circ H$ if and only if $x$ and $y$ are adjacent in $H^*$.
\end{lemma}

\begin{proof}
By Theorem \ref{basictoolLexicographic} (iii),  $ d_{G\circ H}((a,x),(a,y))\le 2$ and, by Theorem \ref{basictoolLexicographic} (i), if $c\ne a$, then  $(c,w)\in N_{G\circ H}(a,x)$ if and only if $c\in N_{G}(a)$. Hence,  $(a,x)$ and $(a,y)$ are MMD if and only if either $(a,x)$ and $(a,y)$ are true twins in $G\circ H$ or $(a,x)$ and $(a,y)$ are not adjacent in $G\circ H$.

On one hand, by the definition of the lexicographic product, $(a,x)$ and $(a,y)$ are not adjacent in $G\circ H$ if and only if $x$ and $y$ are not adjacent in $ H$.

Moreover, by Theorem \ref{basictoolLexicographic} (i),  $(a,x)$ and $(a,y)$ are true twins in $G\circ H$ if and only if $x$ and $y$ are true twins in $H$.

Therefore, the result follows.
\end{proof}


\begin{proposition}{\rm \cite{Kuziak2014}}\label{SR-graph-twins-free}
Let $G$ be a connected  graph of order $n\ge 2$ and let $H$ be a noncomplete graph of order $n'\ge 2$.
If $G$ has no true twin vertices, then
$$(G\circ H)_{SR}\cong \left(G_{SR}\circ H^*\right) \cup \bigcup_{i=1}^{n-|\partial(G)|} H^*_-. $$
\end{proposition}

\begin{proof}
We assume that $G$ has no true twin vertices. By Lemmas \ref{lem mmd} and \ref{lem mmd in copy H}, we have the following facts.
\begin{itemize}
\item For any $a\not\in \partial (G)$ it follows that $(G\circ H)_{SR}$ has a subgraph, say $H_a$, induced by $(\{a\}\times V(H))\cap \partial(G\circ H)$ which is isomorphic to $H^*_-$
\item For any $b\in \partial (G)$, we have that $(G\circ H)_{SR}$ has a subgraph, say $H_b$, induced by $(\{b\}\times V(H))\cap \partial (G\circ H)$ which is isomorphic to $H^* $.
\item The set $(\partial(G)\times V(H))\cap \partial (G\circ H)$ induces a subgraph in $(G\circ H)_{SR}$ which is isomorphic to $G_{SR}\circ H^*$.
\item For any $a\not\in \partial (G)$ and any $b\in \partial (G)$ there are no edges of $(G\circ H)_{SR}$ joining vertices belonging to $H_a$ with vertices belonging to $H_b$.
\item For any distinct vertices $a_1,a_2\not\in \partial (G)$ there are no edges of $(G\circ H)_{SR}$ joining vertices belonging to $H_{a_1}$ with vertices belonging to $H_{a_2}$.
\end{itemize}
Therefore, the result follows.
\end{proof}

Figure \ref{no true twins} shows the graph $P_4\circ P_3$ and its strong resolving graph. Notice that $(P_3)_-^*\cong K_2$, $(P_3)^*\cong K_2\cup K_1$ and $(P_4)_{SR}\cong K_2$. So, $(P_4\circ P_3)_{SR}\cong K_2\circ (K_2\cup K_1)\cup K_2\cup K_2$.

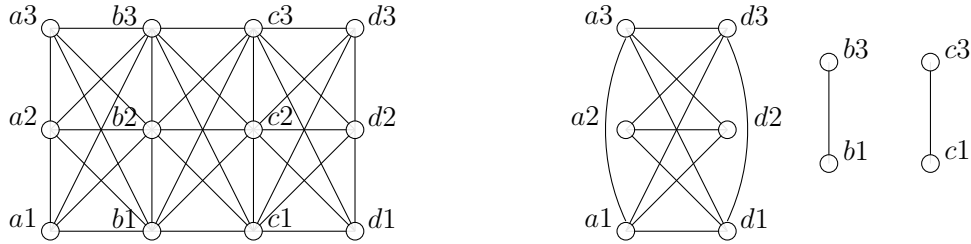
\begin{figure}[ht]
\centering
\begin{tikzpicture}[scale=.9, transform shape]
\draw(0,0)--(0,3);
\draw(1.5,0)--(1.5,3);
\draw(3,0)--(3,3);
\draw(4.5,0)--(4.5,3);
\draw(0,0)--(4.5,0);
\draw(0,1.5)--(4.5,1.5);
\draw(0,3)--(4.5,3);

\draw(0,0)--(1.5,1.5);
\draw(0,0)--(1.5,3);

\draw(0,1.5)--(1.5,0);
\draw(0,1.5)--(1.5,3);

\draw(0,3)--(1.5,0);
\draw(0,3)--(1.5,1.5);

\draw(1.5,0)--(3,1.5);
\draw(1.5,0)--(3,3);

\draw(1.5,1.5)--(3,0);
\draw(1.5,1.5)--(3,3);

\draw(1.5,3)--(3,0);
\draw(1.5,3)--(3,1.5);

\draw(3,0)--(4.5,1.5);
\draw(3,0)--(4.5,3);

\draw(3,1.5)--(4.5,0);
\draw(3,1.5)--(4.5,3);

\draw(3,3)--(4.5,0);
\draw(3,3)--(4.5,1.5);

\filldraw[fill opacity=0.9,fill=white]  (0,0) circle (0.13cm);
\filldraw[fill opacity=0.9,fill=white]  (0,1.5) circle (0.13cm);
\filldraw[fill opacity=0.9,fill=white]  (0,3) circle (0.13cm);
\filldraw[fill opacity=0.9,fill=white]  (1.5,0) circle (0.13cm);
\filldraw[fill opacity=0.9,fill=white]  (1.5,1.5) circle (0.13cm);
\filldraw[fill opacity=0.9,fill=white]  (1.5,3) circle (0.13cm);
\filldraw[fill opacity=0.9,fill=white]  (3,0) circle (0.13cm);
\filldraw[fill opacity=0.9,fill=white]  (3,1.5) circle (0.13cm);
\filldraw[fill opacity=0.9,fill=white]  (3,3) circle (0.13cm);
\filldraw[fill opacity=0.9,fill=white]  (4.5,0) circle (0.13cm);
\filldraw[fill opacity=0.9,fill=white]  (4.5,1.5) circle (0.13cm);
\filldraw[fill opacity=0.9,fill=white]  (4.5,3) circle (0.13cm);
\node at (-0.4,0.2) {$a1$ };
\node at (-0.4,1.7) {$a2$ };
\node at (-0.4,3.2) {$a3$ };
\node at (1.1,0.2) {$b1$ };
\node at (1.1,1.7) {$b2$ };
\node at (1.1,3.2) {$b3$ };
\node at (3.4,0.2) {$c1$ };
\node at (3.4,1.7) {$c2$ };
\node at (3.4,3.2) {$c3$ };
\node at (4.9,0.2) {$d1$ };
\node at (4.9,1.7) {$d2$ };
\node at (4.9,3.2) {$d3$ };

\draw(8.5,0)--(10,1.5);
\draw(8.5,0)--(10,3);

\draw(8.5,1.5)--(10,0);
\draw(8.5,1.5)--(10,3);

\draw(8.5,3)--(10,0);
\draw(8.5,3)--(10,1.5);

\draw(11.5,1)--(11.5,2.5);

\draw(13,1)--(13,2.5);

\draw (8.5,0) -- (10,0);
\draw (8.5,1.5) -- (10,1.5);
\draw (8.5,3) -- (10,3);

\filldraw[fill opacity=0.9,fill=white]  (8.5,0) circle (0.13cm);
\filldraw[fill opacity=0.9,fill=white]  (8.5,1.5) circle (0.13cm);
\filldraw[fill opacity=0.9,fill=white]  (8.5,3) circle (0.13cm);
\filldraw[fill opacity=0.9,fill=white]  (10,0) circle (0.13cm);
\filldraw[fill opacity=0.9,fill=white]  (10,1.5) circle (0.13cm);
\filldraw[fill opacity=0.9,fill=white]  (10,3) circle (0.13cm);
\filldraw[fill opacity=0.9,fill=white]  (11.5,1) circle (0.13cm);
\filldraw[fill opacity=0.9,fill=white]  (11.5,2.5) circle (0.13cm);
\filldraw[fill opacity=0.9,fill=white]  (13,1) circle (0.13cm);
\filldraw[fill opacity=0.9,fill=white]  (13,2.5) circle (0.13cm);

\draw (8.5,0.15) .. controls (8.1,1) and (8.1,2) .. (8.5,2.85);
\draw (10,0.15) .. controls (10.4,1) and (10.4,2) .. (10,2.85);

\node at (8.1,0.2) {$a1$ };
\node at (7.9,1.7) {$a2$ };
\node at (8.1,3.2) {$a3$ };
\node at (11.9,1.2) {$b1$ };
\node at (11.9,2.7) {$b3$ };
\node at (10.4,0.2) {$d1$ };
\node at (10.6,1.7) {$d2$ };
\node at (10.4,3.2) {$d3$ };
\node at (13.4,1.2) {$c1$ };
\node at (13.4,2.7) {$c3$ };

\end{tikzpicture}
\caption{The graph $P_4\circ P_3$ and its strong resolving graph}
\label{no true twins}
\end{figure}

\begin{proposition}{\rm \cite{Kuziak2014}}\label{SR G and complete}
For any connected nontrivial graph $G$ of order $n\ge 2$ and any integer $n'\ge 2$,
$$(G\circ K_{n'})_{SR}\cong (G_{SR}\circ K_{n'}) \cup \bigcup_{i=1}^{n-|\partial(G)|} K_{n'}.$$
\end{proposition}
\begin{proof}
Notice that $(K_{n'})^*\cong K_{n'}$ and, by Lemma \ref{lem mmd in copy H}, for any $a\in V(G)$, the subgraph of $(G\circ K_{n'})_{SR}$ induced by $(\{a\}\times V(K_{n'}))\cap \partial(G\circ K_{n'})$ is isomorphic to $K_{n'}$. Also, from Lemmas \ref{lem mmd} and \ref{lemmaTrueTwins}, the subgraph of $(G\circ K_{n'})_{SR}$ induced by $(\partial(G)\times V(K_{n'}))\cap \partial (G\circ K_{n'})$ is isomorphic to $G_{SR}\circ K_{n'}$. Moreover, for $a\not\in \partial (G)$ and $b\in \partial (G)$ there are not edges of $(G\circ K_{n'})_{SR}$ joining vertices belonging to $\{a\}\times V(K_{n'})$ with vertices belonging to $\{b\}\times V(K_{n'})$. Therefore, the result follows.
\end{proof}

We have studied the case in which the second factor in the lexicographic product is a complete graph. Since this product is not commutative, we now consider the case in which the first factor is a complete graph.

\begin{proposition}{\rm \cite{Kuziak2014}}\label{SR-Completegraph}
Let $n\ge 2$ be an integer and let $H$ be a graph of order $n'\ge 2$.
If $H$ has maximum degree $\Delta(H)\le n'-2$, then
$$(K_n\circ H)_{SR}\cong  \bigcup_{i=1}^n H^*. $$
\end{proposition}

\begin{proof}
We assume that $H$ has maximum degree $\Delta(H)\le n'-2$.  Notice that $H^*$ has no isolated vertices and, by Lemma \ref{lem mmd in copy H}, for any $a\in V(K_n)$, the subgraph $(K_n\circ H)_{SR}$ induced by $(\{a\}\times V(H))\cap \partial(K_n\circ H)$ is isomorphic to $H^*$.

Also, by Lemma \ref{lemmaTrueTwins}, for any distinct $a,b\in V(K_n)$ and any $x,y\in V(H)$, the vertices $(a,x)$ and $(b,y)$ are not MMD in $K_n\circ H$. Therefore, the result follows.
\end{proof}

We define the {\em TF-boundary} of a  noncomplete graph $G=(V,E)$ as a set $\partial_{TF}(G) \subseteq \partial(G)$, where $x\in \partial_{TF}(G)$ whenever there exists $y\in \partial (G)$, such that $x$ and $y$ are MMD in $G$ and $N_G[x]\ne N_G[y]$ (which means that $x,y$ are not true twins).
The \emph{strong resolving TF-graph} of $G$ is a graph $G_{SRS}$  with vertex set $V(G_{SRS}) = \partial_{TF}(G)$, where two vertices $u,v$ are adjacent in $G_{SRS}$ if and only if $u$ and $v$ are MMD in $G$ and $N_G[x]\ne N_G[y]$. Since the strong resolving TF-graph is a subgraph of the strong resolving graph, an instance of the problem of transforming a graph into its strong resolving TF-graph forms part of the general problem of transforming a graph into its strong resolving graph. From \cite{Oellermann2007}, it is known that this general transformation is polynomial. Thus, the problem of  transforming a graph into its strong resolving TF-graph is also polynomial.

\begin{proposition}{\rm \cite{Kuziak2014}}\label{SR-MaxDgn'-2}
Let $G$ be a connected noncomplete   graph of order $n\ge 2$ and let $H$ be a graph of order $n'\ge 2$.
If $H$ has maximum degree $\Delta(H)\le n'-2$, then
$$(G\circ H)_{SR}\cong (G_{SRS}\circ H^*) \cup \bigcup_{i=1}^{n-|\partial_{TF}(G)|} H^*. $$
\end{proposition}

\begin{proof}
We assume that $H$ has maximum degree $\Delta(H)\le n'-2$.  Notice that $H^*$ has no isolated vertices and, by Lemma  \ref{lem mmd in copy H}, for any $a\in V(G)$, the subgraph $(G\circ H)_{SR}$ induced by $(\{a\}\times V(H))\cap \partial(G\circ H)$ is isomorphic to $H^*$.

Also, by Lemma \ref{lemmaTrueTwins}, if two distinct vertices $a,b$ are true twins in $G$ and $x,y\in V(H)$, then $(a,x)$ and $(b,y)$ are not MMD in $G\circ H$. So, from Lemmas \ref{lem mmd} and \ref{lem mmd in copy H} we deduce that the subgraph of $(G\circ H)_{SR}$ induced by $(\partial_{TF}(G)\times V(H))\cap \partial (G\circ H)$ is isomorphic to $G_{SRS}\circ H^*$. Moreover, for $a\not\in \partial_{TF} (G)$ and $b\in \partial_{TF} (G)$ there are no edges of $(G\circ H)_{SR}$ joining vertices belonging to $\{a\}\times V(H)$ with vertices belonging to $\{b\}\times V(H)$. Therefore,
the result follows.
\end{proof}

Figure \ref{max degree} shows the graph $(K_1+(K_1\cup K_2))\circ P_4$ and its strong resolving graph. Notice that $(P_4)^*\cong P_4$ and $(K_1+(K_1\cup K_2))_{SRS} \cong P_3$. So, $((K_1+(K_1\cup K_2))\circ P_4)_{SR}\cong (P_3\circ P_4)\cup P_4$.

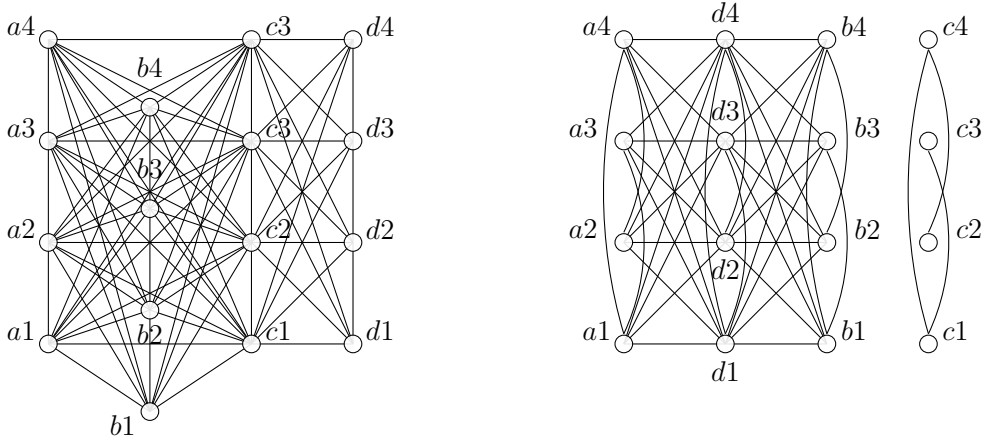
\begin{figure}[h]
\centering
\begin{tikzpicture}[scale=.9, transform shape]
\draw(0,0)--(0,4.5);
\draw(1.5,-1)--(1.5,3.5);
\draw(3,0)--(3,4.5);
\draw(4.5,0)--(4.5,4.5);
\draw(0,0)--(4.5,0);
\draw(0,0)--(1.5,-1);
\draw(1.5,-1)--(3,0);
\draw(0,1.5)--(4.5,1.5);
\draw(0,1.5)--(1.5,0.5);
\draw(1.5,0.5)--(3,1.5);
\draw(0,3)--(4.5,3);
\draw(0,3)--(1.5,2);
\draw(1.5,2)--(3,3);
\draw(0,4.5)--(4.5,4.5);
\draw(0,4.5)--(1.5,3.5);
\draw(1.5,3.5)--(3,4.5);

\draw(0,0)--(1.5,0.5);
\draw(0,0)--(1.5,2);
\draw(0,0)--(1.5,3.5);

\draw(0,1.5)--(1.5,-1);
\draw(0,1.5)--(1.5,2);
\draw(0,1.5)--(1.5,3.5);

\draw(0,3)--(1.5,-1);
\draw(0,3)--(1.5,0.5);
\draw(0,3)--(1.5,3.5);

\draw(0,4.5)--(1.5,-1);
\draw(0,4.5)--(1.5,0.5);
\draw(0,4.5)--(1.5,2);

\draw(0,0)--(3,1.5);
\draw(0,0)--(3,3);
\draw(0,0)--(3,4.5);

\draw(0,1.5)--(3,0);
\draw(0,1.5)--(3,3);
\draw(0,1.5)--(3,4.5);

\draw(0,3)--(3,0);
\draw(0,3)--(3,1.5);
\draw(0,3)--(3,4.5);

\draw(0,4.5)--(3,0);
\draw(0,4.5)--(3,1.5);
\draw(0,4.5)--(3,3);

\draw(1.5,-1)--(3,1.5);
\draw(1.5,-1)--(3,3);
\draw(1.5,-1)--(3,4.5);

\draw(1.5,0.5)--(3,0);
\draw(1.5,0.5)--(3,3);
\draw(1.5,0.5)--(3,4.5);

\draw(1.5,2)--(3,0);
\draw(1.5,2)--(3,1.5);
\draw(1.5,2)--(3,4.5);

\draw(1.5,3.5)--(3,0);
\draw(1.5,3.5)--(3,1.5);
\draw(1.5,3.5)--(3,3);

\draw(3,0)--(4.5,1.5);
\draw(3,0)--(4.5,3);
\draw(3,0)--(4.5,4.5);

\draw(3,1.5)--(4.5,0);
\draw(3,1.5)--(4.5,3);
\draw(3,1.5)--(4.5,4.5);

\draw(3,3)--(4.5,0);
\draw(3,3)--(4.5,1.5);
\draw(3,3)--(4.5,4.5);

\draw(3,4.5)--(4.5,0);
\draw(3,4.5)--(4.5,1.5);
\draw(3,4.5)--(4.5,3);

\filldraw[fill opacity=0.9,fill=white]  (0,0) circle (0.13cm);
\filldraw[fill opacity=0.9,fill=white]  (0,1.5) circle (0.13cm);
\filldraw[fill opacity=0.9,fill=white]  (0,3) circle (0.13cm);
\filldraw[fill opacity=0.9,fill=white]  (0,4.5) circle (0.13cm);
\filldraw[fill opacity=0.9,fill=white]  (1.5,-1) circle (0.13cm);
\filldraw[fill opacity=0.9,fill=white]  (1.5,0.5) circle (0.13cm);
\filldraw[fill opacity=0.9,fill=white]  (1.5,2) circle (0.13cm);
\filldraw[fill opacity=0.9,fill=white]  (1.5,3.5) circle (0.13cm);
\filldraw[fill opacity=0.9,fill=white]  (3,0) circle (0.13cm);
\filldraw[fill opacity=0.9,fill=white]  (3,1.5) circle (0.13cm);
\filldraw[fill opacity=0.9,fill=white]  (3,3) circle (0.13cm);
\filldraw[fill opacity=0.9,fill=white]  (3,4.5) circle (0.13cm);
\filldraw[fill opacity=0.9,fill=white]  (4.5,0) circle (0.13cm);
\filldraw[fill opacity=0.9,fill=white]  (4.5,1.5) circle (0.13cm);
\filldraw[fill opacity=0.9,fill=white]  (4.5,3) circle (0.13cm);
\filldraw[fill opacity=0.9,fill=white]  (4.5,4.5) circle (0.13cm);

\node at (-0.4,0.2) {$a1$ };
\node at (-0.4,1.7) {$a2$ };
\node at (-0.4,3.2) {$a3$ };
\node at (-0.4,4.7) {$a4$ };
\node at (1.1,-1.2) {$b1$ };
\node at (1.5,0.15) {$b2$ };
\node at (1.5,2.6) {$b3$ };
\node at (1.5,4.1) {$b4$ };
\node at (3.4,0.2) {$c1$ };
\node at (3.4,1.7) {$c2$ };
\node at (3.4,3.2) {$c3$ };
\node at (3.4,4.7) {$c3$ };
\node at (4.9,0.2) {$d1$ };
\node at (4.9,1.7) {$d2$ };
\node at (4.9,3.2) {$d3$ };
\node at (4.9,4.7) {$d4$ };

\draw(8.5,0)--(11.5,0);
\draw(8.5,1.5)--(11.5,1.5);
\draw(8.5,3)--(11.5,3);
\draw(8.5,4.5)--(11.5,4.5);

\draw(8.5,0)--(10,1.5);
\draw(8.5,0)--(10,3);
\draw(8.5,0)--(10,4.5);

\draw(8.5,1.5)--(10,0);
\draw(8.5,1.5)--(10,3);
\draw(8.5,1.5)--(10,4.5);

\draw(8.5,3)--(10,0);
\draw(8.5,3)--(10,1.5);
\draw(8.5,3)--(10,4.5);

\draw(8.5,4.5)--(10,0);
\draw(8.5,4.5)--(10,1.5);
\draw(8.5,4.5)--(10,3);

\draw(11.5,0)--(10,1.5);
\draw(11.5,0)--(10,3);
\draw(11.5,0)--(10,4.5);

\draw(11.5,1.5)--(10,0);
\draw(11.5,1.5)--(10,3);
\draw(11.5,1.5)--(10,4.5);

\draw(11.5,3)--(10,0);
\draw(11.5,3)--(10,1.5);
\draw(11.5,3)--(10,4.5);

\draw(11.5,4.5)--(10,0);
\draw(11.5,4.5)--(10,1.5);
\draw(11.5,4.5)--(10,3);


\filldraw[fill opacity=0.9,fill=white]  (8.5,0) circle (0.13cm);
\filldraw[fill opacity=0.9,fill=white]  (8.5,1.5) circle (0.13cm);
\filldraw[fill opacity=0.9,fill=white]  (8.5,3) circle (0.13cm);
\filldraw[fill opacity=0.9,fill=white]  (8.5,4.5) circle (0.13cm);
\filldraw[fill opacity=0.9,fill=white]  (10,0) circle (0.13cm);
\filldraw[fill opacity=0.9,fill=white]  (10,1.5) circle (0.13cm);
\filldraw[fill opacity=0.9,fill=white]  (10,3) circle (0.13cm);
\filldraw[fill opacity=0.9,fill=white]  (10,4.5) circle (0.13cm);
\filldraw[fill opacity=0.9,fill=white]  (11.5,0) circle (0.13cm);
\filldraw[fill opacity=0.9,fill=white]  (11.5,1.5) circle (0.13cm);
\filldraw[fill opacity=0.9,fill=white]  (11.5,3) circle (0.13cm);
\filldraw[fill opacity=0.9,fill=white]  (11.5,4.5) circle (0.13cm);
\filldraw[fill opacity=0.9,fill=white]  (13,0) circle (0.13cm);
\filldraw[fill opacity=0.9,fill=white]  (13,1.5) circle (0.13cm);
\filldraw[fill opacity=0.9,fill=white]  (13,3) circle (0.13cm);
\filldraw[fill opacity=0.9,fill=white]  (13,4.5) circle (0.13cm);

\draw (8.5,0.15) .. controls (8.1,1.5) and (8.1,3) .. (8.5,4.35);
\draw (8.5,0.15) .. controls (8.9,1) and (8.9,2) .. (8.5,2.85);
\draw (8.5,1.65) .. controls (8.9,2.5) and (8.9,3.5) .. (8.5,4.35);
\draw (10,0.15) .. controls (9.6,1.5) and (9.6,3) .. (10,4.35);
\draw (10,0.15) .. controls (10.4,1) and (10.4,2) .. (10,2.85);
\draw (10,1.65) .. controls (10.4,2.5) and (10.4,3.5) .. (10,4.35);
\draw (11.5,0.15) .. controls (11.1,1.5) and (11.1,3) .. (11.5,4.35);
\draw (11.5,0.15) .. controls (11.9,1) and (11.9,2) .. (11.5,2.85);
\draw (11.5,1.65) .. controls (11.9,2.5) and (11.9,3.5) .. (11.5,4.35);
\draw (13,0.15) .. controls (12.6,1.5) and (12.6,3) .. (13,4.35);
\draw (13,0.15) .. controls (13.4,1) and (13.4,2) .. (13,2.85);
\draw (13,1.65) .. controls (13.4,2.5) and (13.4,3.5) .. (13,4.35);

\node at (8.1,0.2) {$a1$ };
\node at (7.9,1.7) {$a2$ };
\node at (7.9,3.2) {$a3$ };
\node at (8.1,4.7) {$a4$ };
\node at (10,-0.4) {$d1$ };
\node at (10,1.1) {$d2$ };
\node at (10,3.45) {$d3$ };
\node at (10,4.9) {$d4$ };
\node at (11.9,0.2) {$b1$ };
\node at (12.1,1.7) {$b2$ };
\node at (12.1,3.2) {$b3$ };
\node at (11.9,4.7) {$b4$ };
\node at (13.4,0.2) {$c1$ };
\node at (13.6,1.7) {$c2$ };
\node at (13.6,3.2) {$c3$ };
\node at (13.4,4.7) {$c4$ };
\end{tikzpicture}
\caption{The graph $(K_1+(K_1\cup K_2))\circ P_4$ and its strong resolving graph}
\label{max degree}
\end{figure}

\subsection{Corona product graphs}

The structure of the strong resolving graph of the corona product can be easily described. By equation \eqref{DistanceCoronaGraphs}, that shows the distance between vertices in the corona product, we deduce that $\partial(G\odot H)=\bigcup_{i=1}^nV_i$ and two vertices $x,y$ are adjacent in $(G\odot H)_{SR}$ if and only if either $x\in V_i$ and $y\in V_j$, where $i\ne j$, or $x,y\in V_i$ and they are adjacent in $H^*$, that is, they are true twins in $H$ or they are adjacent in $H^c$. Therefore $(G\odot H)_{SR}$ is obtained from the complete graph of vertex set $\partial(G\odot H)=\bigcup_{i=1}^nV_i$ by removing the edges  of each copy of $H$ connecting two non-true twin vertices. So we can deduce the following Corollary.

\begin{corollary}
Let $G$, $H$ be two graphs, then $(G\odot H)_{SR}$ is a complete graph if and only if $H$ is either a complete graph or an empty graph.
\end{corollary}

An interesting example of a strong resolving TF-graph defined in Section \ref{SectionLexicographic} can be obtained from the corona graph $G\odot K_{n'}$, $n'\ge 2$, where $G$ has order $n\ge 2$. Notice that any two distinct vertices belonging to any two copies of the complete graph $K_{n'}$ are MMD, but if they are in the same copy, then they are also true twins. Thus, in this case $\partial_{TF}(G\odot K_{n'})=\partial(G\odot K_{n'})$, while we have that $(G\odot K_{n'})_{SR}\cong K_{nn'}$ and $(G\odot K_{n'})_{SRS}$ is isomorphic to a complete $n$-partite graph $K_{n',n',\dots ,n'}$.

\section{Open problems}

The strong resolving graph $G_{SR}$ of a graph $G$ is still not enough known as an interesting and very useful construction. In this sense, some of the next problems would be worthwhile to be dealt with.

\begin{itemize}
  \item It is already known that constructing the strong resolving graph $G_{SR}$ of a graph $G$ can be done in polynomial time. However, not much is known on deciding whether a given graph $H$ is the strong resolving graph of a graph $G$. Some partial results are given in this work, but still much more is required to get a complete characterization.

  \item Is it possible to describe some properties of the strong resolving graph $G_{SR}$ based on some properties of the graph $G$? Can we state for instance whether $G_{SR}$ is connected, bipartite or hamiltonian? Can we assert which is the diameter or the girth of $G_{SR}$ based on some properties of $G$?

  \item Is it possible to characterize the family of graphs $G$ for which $G_{SR}\cong G$?

  \item Given a graph $G$, is it possible to find all the graph $H$ such that $H_{SR}\cong G$?

  \item Is there any other usefulness of the strong resolving graph distinct from that of computing the strong metric dimension?
\end{itemize}

\bibliographystyle{elsart-num-sort}



\end{document}